\documentclass[graybox]{svmult}

\usepackage{mathptmx}        
\usepackage{helvet}          
\usepackage{courier}         
\usepackage{type1cm}         

\usepackage{makeidx}         
\usepackage{graphicx}        
\usepackage{multicol}        
\usepackage[bottom]{footmisc}

\usepackage{amsmath,amssymb,bbm,bm}
\usepackage{graphicx}
\usepackage[utf8]{inputenc}
\usepackage{url}
\usepackage{microtype}
\usepackage{import}
\usepackage{color}
\usepackage{epic}
\usepackage{tikz}
\usepackage{etex}
\usepackage{pgfplots}
\usepackage{subfig}
\usepackage[ruled,vlined]{algorithm2e}
\usepackage[normalem]{ulem}
\usepackage{mathtools}
\usepackage{paralist}

\DeclareMathOperator*{\essinf}{\operatorname{ess\,inf}}

\DeclareMathOperator*{\argmax}{arg\,max}

\newcommand{\mc}{\mathcal}

\newcommand{\ve}{\mathbf e}
\newcommand{\vi}{\mathbf i}
\newcommand{\vj}{\mathbf j}
\newcommand{\vk}{\mathbf k}

\newcommand{\bbN}{\mathbb{N}}
\newcommand{\bbP}{\mathbb{P}}
\newcommand{\bbR}{\mathbb{R}}

\newcommand{\vtau}{\boldsymbol{\tau}}
\newcommand{\vxi}{\boldsymbol{\xi}}

\newcommand{\ev}[1]{ {\boldsymbol{\mathsf  E}} \left[  #1 \right]}
\DeclareMathOperator{\Cov}{Cov}

\renewcommand{\d}{\mathrm{d}}

\newcommand{\Span}{\mathrm{span}}

\DeclareMathAlphabet{\pazocal}{OMS}{zplm}{m}{n}
\newcommand{\CovOp}{C}
\newcommand{\contFun}{\pazocal{C}}

\newcommand{\rev}[1]{\textcolor{black}{#1}}

\begin{document}

%
%
%
%

\title*{On Expansions and Nodes for Sparse Grid Collocation of Lognormal Elliptic PDEs}
\titlerunning{Expansions and nodes for sparse grid collocation for lognormal elliptic PDE}
\author{Oliver~G.~Ernst, Bj\"orn Sprungk, and Lorenzo Tamellini}
\institute{Oliver~G.~Ernst \at Department of Mathematics, TU Chemnitz, Germany, \email{oernst@math.tu-chemnitz.de}
\and 
Bj\"orn Sprungk \at Faculty of Mathematics and Computer Science, TU Bergakademie Freiberg, Germany, \email{bjoern.sprungk@math.tu-freiberg.de}
\and
Lorenzo Tamellini \at Istituto di Matematica Applicata e Tecnologie Informatiche ``E. Magenes'', Pavia, Consiglio Nazionale delle Ricerche, Italy, \email{tamellini@imati.cnr.it}
}

\maketitle

\abstract{
This work is a follow-up \rev{to our} previous contribution (``Convergence of sparse collocation for functions of
countably many Gaussian random variables (with application to elliptic
PDEs)'', \rev{SIAM J. Numer. Anal.}, 2018), and contains further
insights on some aspects of the solution of elliptic PDEs with
lognormal diffusion coefficients using sparse grids.  Specifically, we
first focus on the choice of univariate interpolation rules,
advocating the use of Gaussian Leja points as introduced by Narayan
and Jakeman (``Adaptive Leja sparse grid constructions for
stochastic collocation and high-dimensional approximation'', \rev{SIAM J. Sci. Comput., 2014}) and then discuss the possible computational advantages of replacing the standard Karhunen-Lo\`eve
expansion of the diffusion coefficient with the L\'evy-Ciesielski
expansion, motivated by theoretical work of Bachmayr, Cohen, DeVore,
and Migliorati (``Sparse polynomial approximation of parametric elliptic PDEs. part II: lognormal coefficients'', \rev{ESAIM: M2AN}, 2016).
\rev{Our numerical results indicate that, for the problem under consideration, Gaussian Leja collocation points outperform Gauss--Hermite and Genz--Keister nodes for the sparse grid approximation and that the Karhunen–Lo\`eve expansion of the log diffusion coefficient is more appropriate than its L\'evy–Ciesielski expansion for purpose of sparse grid collocation.}
}

\section{Introduction} \label{sec:intro}
We consider the sparse polynomial collocation method for approximating the solution of a random elliptic boundary value problem with lognormal diffusion coefficient, a well-studied model problem for uncertainty quantification in numerous physical systems such as stationary groundwater flow in an uncertain aquifer.
The assumption of a lognormal diffusion coefficient, i.e., that its logarithm is a Gaussian random field, \rev{is a common, quite simple approach for modeling uncertain conductivities with large variability in practice (a discussion on this and other, more sophisticated models for the conductivity of aquifers can be found e.g. in \cite{neuman.riva.guad:trunc.power}, empirical evidence for lognormality is discussed in \cite{Freeze1975}),
but already yields an interesting setting from a mathematical  point of view.
For instance, a lognormal diffusion coefficient} introduces challenges, e.g., for stochastic Galerkin methods \cite{gittelson:logn,sarkis:lognormal,MuglerStarkloff2013} due to the unboundedness of the coefficient and the necessity of solving large coupled linear systems. 
By contrast, stochastic collocation based on sparse grids \cite{XiuHesthaven2005,BabuskaEtAl2010,NobileEtAl2008a,NobileEtAl2008b} has been established as a powerful and flexible non-intrusive approximation method in high dimensions for functions of weighted mixed Sobolev regularity.
The fact that solutions of lognormal diffusion problems belong to this function class has been shown under suitable assumptions in \cite{BachmayrEtAl2015}.
Based on the analysis in \cite{BachmayrEtAl2015}, we have established in \cite{ErnstEtAl2018} a dimension-independent convergence rate for sparse polynomial collocation given a mild condition on the univariate node sets.
This condition is, for instance, satisfied by the classical Gauss-Hermite nodes \cite{ErnstEtAl2018}.
In related work, dimension-independent convergence has also been shown for sparse grid quadrature \cite{Chen2016}.

This work is a follow-up on our previous contribution \cite{ErnstEtAl2018} and provides further discussion, insights and numerical results concerning two important design decisions for sparse polynomial collocation applied to differential equations with Gaussian random fields.\par

The first concerns the representation of the Gaussian random field by a series expansion.
A common choice is to use the Karhunen-Lo\`eve expansion \cite{GhanemSpanos1991} of the random field.
Although it represents the spectral, and thus \rev{$L^2$-}optimal, expansion of the input field, it is not necessarily the \rev{most efficient parametrization for approximating} the solution field of the equation. 
In particular, in \cite{BachmayrEtAl2015, BachmayrEtAl2018} the authors \rev{advocate} using wavelet-based expansions with localized basis functions.
A classical example of this type is the L\'evy-Ciesielski (LC) expansion of Brownian motion or a Brownian bridge \cite{Ciesielski1961, BhattacharyaWaymire2016}, which employs hat functions, \rev{whereas the KL expansion of the same random fields results in sinusoidal (hence smoother and globally supported) basis functions.}
A theoretical advantage of localized expansions of Gaussian random fields is that for these it is easier to verify the (sufficient) condition for weighted mixed Sobolev regularity of the solution of the associated lognormal diffusion problem.
In this work, we conduct numerical experiments with the KL and LC expansions of a Brownian bridge as the lognormal coefficient in an elliptic diffusion equation in order to study their relative merits for sparse \rev{grid} collocation of the resulting solution.
We note that finding optimal representations of the random inputs is a topic of ongoing research, see e.g. \cite{Bohn2018,Papaioannou2019,Tipireddy2014}.

The second design decision we investigate is the choice of the univariate polynomial interpolation \rev{node} sequences which form the building blocks of sparse \rev{grid} collocation.
Established schemes are Lagrange interpolation based on Gauss--Hermite or Genz--Keister nodes.
However, the former are non-nested and the latter grow rapidly in number and are only available up to a certain level.
In recent work, weighted Leja nodes \cite{NarayanJakeman2014} have been advocated as a suitable nested and slowly increasing node family for sparse grid approximations, see, e.g., \cite{Loukrezis2019,FarcasEtAl2019, VanDenBosEtal2018} for recent applications in uncertainty quantification.
However, \rev{so far there exist only preliminary results regarding the numerical analysis of weighted Leja points on unbounded domains}, e.g., \cite{JantschEtAl2019}.
We provide numerical evidence that \rev{Gaussian Leja nodes, i.e., weighted Leja nodes with Gaussian weight, satisfy as well} the sufficient condition given in \cite{ErnstEtAl2018} for dimension-independent sparse polynomial collocation.
Moreover, we compare the performance of sparse grid collocation based on Gaussian Leja, Gauss--Hermite and Genz--Keister nodes for the approximation of the solution of a lognormal random diffusion equation.

The remainder of the paper is organized as follows.
In Section \ref{sec:pde} we provide the necessary fundamentals on lognormal diffusion problems and discuss the classical Karhunen--Lo\`eve expansion of random fields and expansions based on wavelets.
Sparse polynomial collocation using sparse grids are introduced in Section \ref{sec:sparse_grids}, where we also recall our convergence results from \cite{ErnstEtAl2018}.
Moreover, we discuss the use of Gaussian Leja points for quadrature and sparse grid collocation in connection with Gaussian distributions in Section \ref{sec:Comp_Nodes}.
Finally, in Section \ref{sec:numerics}, we present our numerical results for sparse polynomial collocation applied to lognormal diffusion problems using the \rev{above-mentioned} univariate node families and \rev{expansion variants} for random fields.
We draw final conclusions in Section \ref{sec:concl}.

\section{Lognormal Elliptic Partial Differential Equations} \label{sec:pde}

We consider a random elliptic boundary value problem on a bounded domain $D\subset \mathbb R^d$ with smooth boundary $\partial D$,
\begin{equation} \label{equ:PDE}
	-\nabla \cdot (a(\omega)\,\nabla u(\omega))
	= f
	\quad \text{in } D , 
	\qquad
	u(\omega) = 0\;\text{on }\partial D,
	\qquad
	\bbP\text{-a.s.~,}
\end{equation}
with a random diffusion coefficient $a:D\times \Omega\to \bbR$ w.r.t.~an underlying probability space  $(\Omega, \mathcal A, \bbP)$.
If $a(\cdot,\omega)\colon D \to \bbR$ satisfies the conditions of the Lax--Milgram lemma  \cite{GilbargTrudinger2001} $\bbP$-almost surely, then a pathwise solution $u:\Omega\to H_0^1(D)$ of \eqref{equ:PDE} exists.
Under suitable assumptions on the integrability of $a_{\min}(\omega):=\essinf_{x\in D} a(x,\omega)$ one can show that $u$ belongs to a Lebesgue--Bochner space \rev{$L^p_\bbP(\Omega; H_0^1(D))$} consisting of all random functions $v\colon \Omega \to H_0^1(D)$ with $\|v\|_{L^p} := \left( \int_\Omega \|v(\omega)\|^p_{H_0^1(D)} \ \bbP(\d \omega)\right)^{1/p}$.

In this paper, we consider \emph{lognormal} random coefficients $a$, i.e., where $\log a\colon D\times \Omega\to\bbR$ is a \emph{Gaussian random field} which is uniquely determined by its mean function $\phi_0\colon D \to \bbR$, $\phi_0(x):= \ev{\log a(x)}$ and its covariance function $c\colon D\times D\to\bbR$,
$c(x,x') := \Cov(\log a(x), \log a(x'))$.
If the Gaussian random field $\log a$ has continuous paths the existence of a weak solution $u\colon \Omega \to H_0^1(D)$ can be ensured.

\begin{proposition}[{\cite[Section 2]{Charrier2012}}] \label{propo:sol_charrier}
Let $\log a$ in \eqref{equ:PDE} be a Gaussian random field with $a(\cdot,\omega)\in \contFun(D)$ almost surely. 
Then a unique solution $u\colon \Omega \to H_0^1(D)$ of \eqref{equ:PDE} exists such that $u \in L^p_\bbP(\Omega; H_0^1(D))$ for any $p >0$.
\end{proposition}

A Gaussian random field $\log a\colon D\times \Omega \to \bbR$ can be represented as a series expansion of the form
\begin{equation} \label{equ:KLE}
	\log a(x, \omega) = \phi_0(x) + \sum_{m\geq1} \phi_m(x) \, \xi_m(\omega), 
	\qquad
	\xi_m\sim \mathsf{N}(0,1) \text{ i.i.d.},
\end{equation}
with suitably chosen $\phi_0, \phi_m \in L^\infty(D)$\rev{, $m\geq 1$}.
In general, several such expansions or expansion bases $\{\phi_m\}_{m\in\bbN}$, respectively, can be constructed, cf. Section \ref{sec:expansions}---thus raising the question of whether certain bases $\{\phi_m\}_{m\in\bbN}$ are better suited for 
parametrizing random fields than others.
Conversely, given an appropriate system $\{\phi_m\}_{m\in\bbN}$, the construction \eqref{equ:KLE} will yield a Gaussian random field if we ensure that the expansion in \eqref{equ:KLE} converges $\bbP$-almost surely pointwise or in $L^\infty(D)$, i.e., that the Gaussian coefficient sequence $(\xi_m)_{m\in\bbN}$ in $\bbR^\bbN$ with distribution $\mu := \bigotimes_{m\in\bbN} N(0,1)$ satisfies
\begin{equation}\label{equ:Gamma}
	\mu(\Gamma) = 1
	\quad \text{ where } \quad
	\Gamma 
	:= 
	\Bigl\{
	\vxi\in\bbR^\bbN:  \| \sum_{m=1}^\infty \phi_m \xi_m\|_{L^\infty(D)} < \infty 
	\Bigr\}.
\end{equation}
We remark that $\Gamma$ is a linear subspace of $\bbR^\bbN$.
The basic condition \eqref{equ:Gamma} is satisfied, for instance, if
\begin{equation} \label{equ:KL_modes}
	\sum_{m\geq1}\|\phi_m\|_{L^\infty(D)} < \infty
\end{equation}
and \eqref{equ:KLE} then yields a Gaussian random variable in $L^\infty(D)$, see \cite[Lemma 2.28]{SchwabGittelson2011} or \cite[Section 2.2.1]{Sprungk2017}.
Given the assumption \eqref{equ:Gamma} we can view the random function $a$ in \eqref{equ:KLE} and the resulting pathwise solution $u$ of \eqref{equ:PDE} as functions in $L^\infty(D)$ and $H_0^1(D)$, respectively, depending on the random parameter $\vxi \in \Gamma$, i.e., $a\colon \Gamma \to L^\infty(D)$ and $u\colon \Gamma \to H_0^1(D)$.
In particular, by the Lax--Milgram lemma we have that $u(\vxi) \in H_0^1(D)$ is well-defined for $\vxi\in\Gamma$ and
\[
	\|u(\vxi)\|_{H_0^1(D)}
	\leq
	\frac {C_D}{a_{\min}(\vxi)} \|f\|_{L^2(D)},
	\qquad
	a_{\min}(\vxi)
	:= \essinf_{x\in D} a(x,\vxi).
\]
In the following subsection, we provide sufficient conditions on the series representation in \eqref{equ:KLE} such that \eqref{equ:Gamma} holds and that the solution $u\colon \Gamma \to H_0^1(D)$ of \eqref{equ:PDE} belongs to a Lebesgue--Bochner space $L^p_\mu(\Gamma; H_0^1(D))$.
Moreover, we discuss the regularity of the solution $u$ of the random PDE \eqref{equ:PDE} as a function of the variable $\vxi\in\Gamma$, which 
\rev{governs approximability} 
by polynomials in $\vxi$. 

\subsection{Integrability and Regularity of the Solution}
A first result concerning the integrability of $u$ given $\log a$ as in \eqref{equ:KLE} is the following.
\begin{proposition}[{\cite[Proposition 2.34]{SchwabGittelson2011}}]\label{propo:sol}
If the functions $\phi_m$, $m\in\bbN$, in \eqref{equ:KLE} satisfy \eqref{equ:KL_modes}, then \eqref{equ:Gamma} holds and the solution $u\colon \Gamma \to H_0^1(D)$ of \eqref{equ:PDE} with diffusion coefficient $a$ as in \eqref{equ:KLE} satisfies $u \in L^p_\mu(\Gamma; H_0^1(D))$ for any $p >0$.
\end{proposition}
In \cite[Corollary 2.1]{BachmayrEtAl2015} the authors establish the same statements as in Proposition \ref{propo:sol} but under the assumption that there exists a strictly positive sequence $(\tau_m)_{m\in\bbN}$ such that
\begin{equation} \label{equ:KL_modes_Bachmayr}
	\sup_{x\in D} \sum_{m\geq1} \tau_m |\phi_m(x)| < \infty, \qquad 
	\sum_{m\geq1} \exp(-\tau^2_m) < \infty.
\end{equation}
Compared with \eqref{equ:KL_modes}, this relaxes the summability condition if the functions $\phi_m$ have local support.
On the other hand, \eqref{equ:KL_modes_Bachmayr} implies that $(|\phi_m(x)|)_{m\in\bbN}$ decays slightly faster than a general $\ell^1(\bbN)$-sequence due to the required growth of $\tau_m \geq C  \sqrt{\log m}$.

The authors of \cite{BachmayrEtAl2015} 
\rev{further establish} 
a particular weighted Sobolev regularity of the solution $u\colon \Gamma \to H_0^1(D)$ of \eqref{equ:PDE} w.r.t.~$\vxi$ or $\xi_m$, respectively, assuming a stronger version of \eqref{equ:KL_modes_Bachmayr}.
To state their result, we introduce further notation.
We define the partial derivative $\partial_{\xi_m} v(\vxi)$ for a function $v\colon \Gamma \to H_0^1(D)$ by 
\[
	\partial_{\xi_m} v(\vxi)
	:=
	\lim_{h\to0}
	\frac{v(\vxi + h \ve_m ) - v(\vxi)}{h},
\] 
when it exists, where $\ve_m$ denotes the $m$-th unit vector in $\bbR^\bbN$.
Higher derivatives $\partial^k_{\xi_m} v(\vxi)$ are defined inductively.
Thus, for any $k\in\bbN$ we have $\partial^k_{\xi_m}v \colon \Gamma \to H_0^1(D)$, assuming its existence on $\Gamma$.
In order to denote arbitrary mixed derivatives we introduce the set
\begin{equation} \label{multi-indices}
	\mc F := \bigl\{\vk \in \bbN_0^\bbN: |\vk|_0 < \infty \bigr\}, 
	\qquad 
	|\vk|_0 := |\{m \in \bbN: k_m > 0 \}|,
\end{equation}
of finitely supported multi-index sequences $\vk \in \bbN_0^\bbN$.
For $\vk \in \mc F$ we can then define the partial derivative $\partial^\vk v\colon \Gamma \to H_0^1(D)$ of a function $v\colon \Gamma \to H_0^1(D)$ by
\[
	\partial^\vk v(\vxi)
	:=
	\left(\prod_{m\geq 1} \partial^{k_m}_{\xi_m}\right) v(\vxi),
\]
where the product is, in fact, finite due to the definition of $\mc F$.

\begin{remark}
It was shown in \cite{BachmayrEtAl2015} that the partial derivative $\partial^\vk u(\vxi) \in H_0^1(D)$, $\vk\in\mc F$, of the solution $u$ of \eqref{equ:PDE} can itself be characterized as the solution of a variational problem in $H_0^1(D)$:
\[
	\int_D 
	a(\vxi)\,\nabla [\partial^\vk u(\vxi)] \cdot \nabla v\ \d x
	= 
	\int_D 
	\sum_{\vi \lneqq \vk} 
	\begin{pmatrix} \vk\\\vi \end{pmatrix}\ 
	\phi^{\vk-\vi} a(\vxi) \nabla [\partial^\vi u(\vxi)] \cdot \nabla v
	\ \d x
	\quad
	\forall v\in H_0^1(D)
\]
where $\vi \lneqq \vk$ denotes that $i_m \leq k_m$ for all $m\in\bbN$ but $\vi\neq\vk$ 
and $\phi^{\vi}$, $\vi \in \mc F$, is a shorthand notation for the finite product $\prod_{m\geq 1}\phi^{i_m}_m\in L^\infty(D)$.
\end{remark}
We now state the regularity result in \cite{BachmayrEtAl2015} which uses a slightly stronger assumption than \eqref{equ:KL_modes_Bachmayr}.

\begin{theorem}[{\cite[Theorem 4.2]{BachmayrEtAl2015}}]\label{theo:Bachmayr_reg}
Let $r\in\bbN$ and let there exist strictly positive weights $\tau_m >0$, $m\in\bbN$ such that for the functions $\phi_m$, $m\in\bbN$, in \eqref{equ:KLE} and for a $p>0$ we have
\begin{align}\label{equ:KL_Bachmayr_2}
	\sup_{x\in D} 
	\sum_{m\geq1} \tau_m |\phi_m(x)| < \frac{\log 2}{\sqrt r}
	\qquad
	\sum_{m\geq1} \tau^{-p}_m < \infty.
\end{align}
Then the solution $u\colon \Gamma \to H_0^1(D)$ of \eqref{equ:PDE} with coefficient a as in \eqref{equ:KLE} satisfies
\begin{equation} \label{equ:bound_partial_weighted}
	\sum_{\substack{\vk\in\mc F, \\ |\vk|_\infty \leq r}} 
	\frac{\vtau^{2\vk}}{\vk!} \|\partial^\vk u\|^2_{L^2_\mu} 
	< 
	\infty,
	\quad
	\text{ where }
	\vtau^{\vk} = \prod_{m\geq1} \tau_m^{k_m}
	\text{ and }
	\vk! = \prod_{m\geq1} k_m! ~.
\end{equation}
\end{theorem}
This theorem tells us that, given \eqref{equ:KL_Bachmayr_2}, the partial derivatives $\partial^\vk u\colon \Gamma \to H_0^1(D)$ exist for any $\vk \in \mc F$ with $|\vk|_\infty < \infty$ and belong to $L^2_\mu(\Gamma; H_0^1(D))$.
Moreover, their $L^2_\mu$-norm decays faster than $\vtau^{-2\vk}$---otherwise \eqref{equ:bound_partial_weighted} would not hold.
In particular, Theorem \ref{theo:Bachmayr_reg} establishes a \emph{weighted 
mixed Sobolev regularity} of the solution $u\colon \Gamma \to H_0^1(D)$ of maximal degree $r\in\bbN$ and with increasing weights $\tau_m \geq C m^{1/p}$.
As it turns out, it is such a regularity which ensures dimension-independent convergence rates for polynomial sparse grid collocation approximations---see the next section.

Moreover, the condition \eqref{equ:KL_Bachmayr_2} seems to favor localized basis functions $\phi_m$ for which $\sum_{m=1}^\infty \tau_m |\phi_m(x)|$ reduces to a summation over a subsequence $\sum_{k=1}^\infty \tau_{m_k} |\phi_{m_k}(x)|$ such that \eqref{equ:KL_Bachmayr_2} is easier to verify.
In view of this, the authors of \cite{BachmayrEtAl2015,BachmayrEtAl2018} proposed using wavelet-based expansions for Gaussian random fields with sufficiently localized $\phi_m$ 
in place of the globally supported eigenmodes $\phi_m$ in the Karhunen--Lo\`eve (KL) expansion. 
In fact, condition \eqref{equ:KL_Bachmayr_2} fails to hold for the KL expansion of some rough Gaussian processes (Example \ref{exam:BB} below), but can be established if the process is sufficiently smooth (Example \ref{exam:smoothBB}).
We will discuss KL and wavelet-based expansions of Gaussian processes in more detail in the next subsection.

\subsection{Choice of Expansion Bases}\label{sec:expansions}

Given a Gaussian random field $\log a\colon D\times \Omega\to\bbR$ with mean $\phi_0\colon D\to\bbR$ and covariance \rev{function} $c\colon D\times D\to\bbR$ we seek a representation as an expansion \eqref{equ:KLE}.
We explain in the following how such expansions can be derived in general.
To this end, we assume that the random field has $\bbP$-almost surely continuous paths, i.e., $\log a\colon \Omega \to \contFun(D)$, and a continuous covariance function $c \in \contFun(D\times D)$.
Thus, we can view $\log a$ also as a Gaussian random variable with values in the separable Banach space $\contFun(D)$ or, by continuous embedding, with values in the separable Hilbert space $L^2(D)$.
The covariance operator $\CovOp \colon L^2(D) \to L^2(D)$ of the random variable $\log a\colon \Omega \to L^2(D)$ is then given by $(\CovOp f)(x) := \int_D c(x,y) \ f(y)\ \d y$.
This operator is \rev{of} trace class and induces a dense subspace $\mc H_{\CovOp} := \mathrm{range}\ \CovOp^{1/2} \subset L^2(D)$, which equipped with the inner product $\langle u,v \rangle_\CovOp := \langle \CovOp^{-1/2} u, \CovOp^{-1/2} v \rangle_{L^2(D)}$, forms again a Hilbert space, called the \emph{Cameron--Martin space} (CMS) of $\log a$.
The CMS plays a crucial role for series representations \eqref{equ:KLE} of $\log a$.
Specifically, it is shown in \cite{LuschgyPages2009} that \eqref{equ:KLE} holds almost surely in $C(D)$ if and only if the system $\{\phi_m\}_{m\in\bbN}$ is a so-called \emph{Parseval frame} or \emph{(super) tight frame} in the CMS of $\log a$, i.e., if $\{\phi_m\}_{m\in\bbN} \subset \mc H_{\CovOp}$ and
\[
	\sum_{m\geq 1} \left| \langle \phi_m, f \rangle_{\CovOp} \right|^2
	=
	\|f\|^2_{\CovOp}
	\qquad
	\forall f \in \mc H_\CovOp.
\]
We discuss two common choices for such frames below.

\paragraph{\bf Karhunen--Lo\`eve expansions.}
This expansion is based on the eigensystem $(\lambda_m,\psi_m)_{m\in\bbN}$ of the compact and self-adjoint covariance operator $\CovOp\colon L^2(D)\to L^2(D)$ of $\log a$.
Thus, let $\psi_m \in L^2(D)$ satisfy $\CovOp\psi_m = \lambda_m \psi_m$ with $\lambda_m > 0$.
Since \rev{the covariance function $c$ is a continuous function on $D\times D$}, we have $\psi_m \in \contFun(D)$ and \eqref{equ:KLE} holds almost surely in $\contFun(D)$ with $\phi_m := \lambda^{-1/2}_m \psi_m$, because $\{\phi_m\}_{m\in\bbN} \subset \mc H_{\CovOp}$ is a complete orthonormal system (CONS) of $\mc H_{\CovOp}$.
In fact, the KL basis $\{\phi_m\}_{m\in\bbN}$ represents the only CONS of $\mc H_{\CovOp}$ \rev{that} is also $L^2(D)$-orthogonal. 
In addition, as the spectral expansion of $\log a$ in $L^2_\bbP(L^2(D))$, it is the optimal basis in this space in the sense that the truncation error after $M$ terms $\|\log a - \phi_0 - \sum_{m=1}^M \phi_m \xi_m\|_{L^2_\bbP(L^2(D))}$ is the smallest among all truncated expansions of length $M$ of the form
\[
	\log a(x,\omega) = \phi_0(x) + \sum_{m=1}^M \tilde \phi_m(x) \tilde \xi_m(\omega).
\]
Under additional assumptions the KL expansion also yields optimal rates of the truncation error in $L^2_\bbP(\contFun(D))$, see again \cite{LuschgyPages2009}.
However, the KL \emph{modes} $\phi_m$ typically have global support on $D$, which often makes it difficult to verify a condition like \eqref{equ:KL_Bachmayr_2}.
Nonetheless, for particular covariance functions, such as the Mat\'ern kernels, bounds on the norms $\|\phi_m\|_{L^\infty(D)}$ are known, see, e.g., \cite{GrahamEtAl2015}.
 
\paragraph{\bf Wavelet-based expansions.}
We now consider expansions in orthonormal wavelet bases $\{\psi_m\}_{m\in\bbN}$ of $L^2(D)$.
Given a factorization $\CovOp = SS^*$, $S\colon L^2(D)\to L^2(D)$, of the covariance operator $\CovOp$ (e.g., $S = S^* = \CovOp^{1/2}$), we can set $\phi_m := S \psi_m$ and obtain a CONS $\{\phi_m\}_{m\in\bbN}$ of the CMS $\mc H_{\CovOp}$, see \cite{LuschgyPages2009}.
Thus, \eqref{equ:KLE} holds almost surely in $\contFun(D)$ with $\phi_m = S \psi_m$.  
The advantage of wavelet-based expansions is that the resulting $\phi_m$ often inherit the localized behavior of the underlying $\psi_m$, cf.  Example \ref{exam:BB}, which then facilitates verification of the sufficient condition \eqref{equ:KL_Bachmayr_2} for the weighted Sobolev regularity of the solution $u$ of \eqref{equ:PDE}.
For instance, we refer to \cite{BachmayrEtAl2018} for Meyer wavelet-based expansions of Gaussian random fields with Mat\'ern covariance functions satisfying \eqref{equ:KL_Bachmayr_2}.
There, the authors use a periodization approach and construct the $\phi_m$ via their Fourier transforms.
Further work on constructing and analyzing wavelet-based expansions of Gaussian random fields includes, e.g., \cite{ElliottMajda1994, ElliottEtAl1997, BenassiEtAl1997}.

\begin{example}[Brownian bridge]\label{exam:BB}
A simple but useful example is the standard \emph{Brownian bridge} $B\colon D\times \Omega \to \bbR$ on $D=[0,1]$.
This is a Gaussian process with mean $\phi_0 \equiv 0$ and covariance function $c(x,x') = \min(x,x') - xx'$.
The associated CMS is given by $\mc H_{\CovOp} = H_0^1(D)$ with $\langle u, v \rangle_{\CovOp} = \langle \nabla u, \nabla v\rangle_{L^2(D)}$ and we have $\CovOp = SS^*$ with
\[
	Sf(x) := \int_0^1 \left( \mathbf{1}_{[0, x]}(y) - x\right) f(y) \, \d y, 
	\qquad f\in L^2(D).
\]
The KL expansion of the Brownian bridge is given by
\begin{equation}\label{equ:BB_KLE}
	B(x,\omega)
	=
	\sum_{m\geq 1} \frac{\sqrt2}{\pi m} \sin(\pi m x) \xi(\omega),
	\qquad
	\xi_m \sim \mathsf{N}(0,1)\text{ i.i.d.~,}
\end{equation}
i.e., we have $\phi_m(x) = \frac{\sqrt2}{\pi m} \sin(\pi m x)$ and $\|\phi_m\|_{L^\infty(D)} = \frac{\sqrt2}{\pi m}$.
Although the \rev{functions} $\phi_m$ do not satisfy the assumptions of Proposition \ref{propo:sol}, existence and integrability of the solution $u$ of \eqref{equ:PDE} for $\log a = B$ is guaranteed by Proposition \ref{propo:sol_charrier}, since $B$ has almost surely continuous paths.
Concerning the condition \eqref{equ:KL_Bachmayr_2} it can be shown that $\sum_{m\geq1} \tau_m |\phi_m(x)| $ converges pointwise to a (discontinuous) function if $\tau_m \in o(m^{-1})$, i.e., $(\tau_m^{-1})_{m\in\bbN} \in \ell^p(\bbN)$ only for a $p > 1$, see the Appendix. 
However, this function turns out to be unbounded in a neighborhood of $x=0$ if $(\tau_m^{-1})_{m\in\bbN} \in \ell^p(\bbN)$ for $p\leq 2$,
and numerical evidence suggests that it is also unbounded if $(\tau_m^{-1})_{m\in\bbN} \in \ell^p(\bbN)$ for $p>2$, again see the Appendix.
Thus, the KL expansion of the Brownian bridge does not satisfy the conditions of Theorem \ref{theo:Bachmayr_reg} for the weighted Sobolev regularity of $u\colon \Gamma \to H_0^1(D)$.

Another classical series representation of the Brownian bridge is the \emph{L\'evy--Ciesielski expansion} \cite{Ciesielski1961}.
This wavelet-based expansions uses the Haar wavelets $\psi_m(x) = 2^{\ell/2} \psi(2^\ell x - j)$ where $\psi(x) = \textbf{1}_{[0,1/2]}(x) - \textbf{1}_{(1/2,1]}(x)$ is the mother wavelet and $m = 2^\ell + j$ for level $\ell \geq 0$ and shift $j=0,\ldots,2^\ell-1$.
Since the Haar wavelets form a CONS of $L^2(D)$ we obtain a Parseval frame of the CMS of the Brownian bridge by taking $\phi_m = S\psi_m$, which yields a Schauder basis consisting of the hat functions
\begin{equation}\label{equ:BB_LCE}
	\phi_m(x) := 2^{-\ell/2} \phi(2^\ell x - j),
	\quad
	\phi(x) := \max(0, 1-|2x-1|), 
	\quad m = 2^\ell + j,
\end{equation}
with $j = 0,\ldots,2^{\ell}-1$ and $\ell\geq0$.
Hence, for $\log a = B$ the series representation \eqref{equ:KLE} also holds almost surely in $\contFun(D)$with $\phi_m$ as in \eqref{equ:BB_LCE}, see also \cite[Section IX.1]{BhattacharyaWaymire2016}.
Moreover, we have $\|\phi_m\|_{L^\infty} = 2^{-\lfloor \log_2 m \rfloor/2}$, resulting in $\sum_{m\geq 1} \|\phi_m\|_{L^\infty} = \infty$.
On the other hand, due to the localization of $\phi_m$ we have \rev{that} for any fixed $x\in D$ and each level $\ell \geq0$ there exists only one $k_\ell \in \{0,\ldots,2^\ell-1\}$ such that $\phi_{2^\ell + k_\ell}(x) \neq 0$.
In particular, it can be shown that the LC expansion of the Brownian bridge satisfies the conditions of Theorem \ref{theo:Bachmayr_reg} for any $p > 2$, since for $\tau_m = \kappa^{\lfloor \log_2 m \rfloor}$ with $|\kappa|< \sqrt2$ we get
\[
	\sum_{m\geq1} \kappa^{\lfloor \log_2 m \rfloor}  |\phi_m(x)|
	=
	\sum_{l\geq0} \kappa^{\ell/2} |\phi_{2^\ell + k_\ell}(x)|
	\leq
	\sum_{l\geq0} (\sqrt{0.5}\rho)^{\ell} 
	< \infty
\]
and for $p > \log_\kappa 2 > 2$
\[
	\sum_{m\geq1}  \tau_m^{-p}
	=
	\sum_{l\geq0}  2^l 	\kappa^{-\ell p} 
	=
	\sum_{l\geq0}  \left( 2\kappa^{-p} \right)^\ell
	< \infty.
\]
\end{example}

\begin{example}[Smoothed Brownian bridge]\label{exam:smoothBB}
Based on the explicit KL expansion of the Brownian bridge we can construct Gaussian random fields with smoother realizations by
\begin{equation}\label{equ:BB_KLE_smooth}
	B_q(x,\omega)
	=
	\sum_{m\geq 1} \frac{\sqrt2}{(\pi m)^q} \sin(\pi m x) \xi(\omega),
	\qquad
	\xi_m \sim \mathsf{N}(0,1)\text{ i.i.d.~,}
	\quad q > 1.
\end{equation}
Now, the resulting $\phi_m = \frac{\sqrt2}{(\pi m)^q} \sin(\pi m \cdot)$ indeed satisfy the assumptions of Proposition \ref{propo:sol} for any $q >1$, since $\|\phi_m\|_{L^\infty(D)} \propto m^{-q}$.
Moreover, for $p>\frac 1{q-1}$ the expansion \eqref{equ:BB_KLE_smooth} satisfies the assumptions of Theorem \ref{theo:Bachmayr_reg} with $\tau_m = m^{(1+\epsilon)/p}$ for sufficiently small $\epsilon$.
For this Gaussian random field $B_q$ the covariance function is given by $c(x,y) = 2 \sum_{m\geq 1} (\pi m)^{-2q} \sin(\pi m x)\sin(\pi m y)$ and we can express $\CovOp^{1/2}$ via
\[
	\CovOp^{1/2} f(x) = \int_D k(x,y)\ f(y)\ \d y,
	\qquad
	k(x,y) = 2\sum_{m\geq 1} (\pi m)^{-q} \sin(\pi m x)\sin(\pi m y).
\]
Thus, we could construct alternative expansion bases for $B_q$ via $\phi_m = \CovOp^{1/2}\psi_m$ given a wavelet CONS $\{\psi_m\}_{m\in\bbN}$ of $L^2(D)$.
However, in this case the resulting $\phi_m$ do not necessarily have a localized support.
For instance, when taking Haar wavelets $\psi_m$ the $\CovOp^{1/2}\psi_m$ have global support in $D=[0,1]$, see Fig. \ref{fig:Phi_Wavelet}.
\end{example}

\begin{figure}
\centering
\includegraphics[width=.3\textwidth]{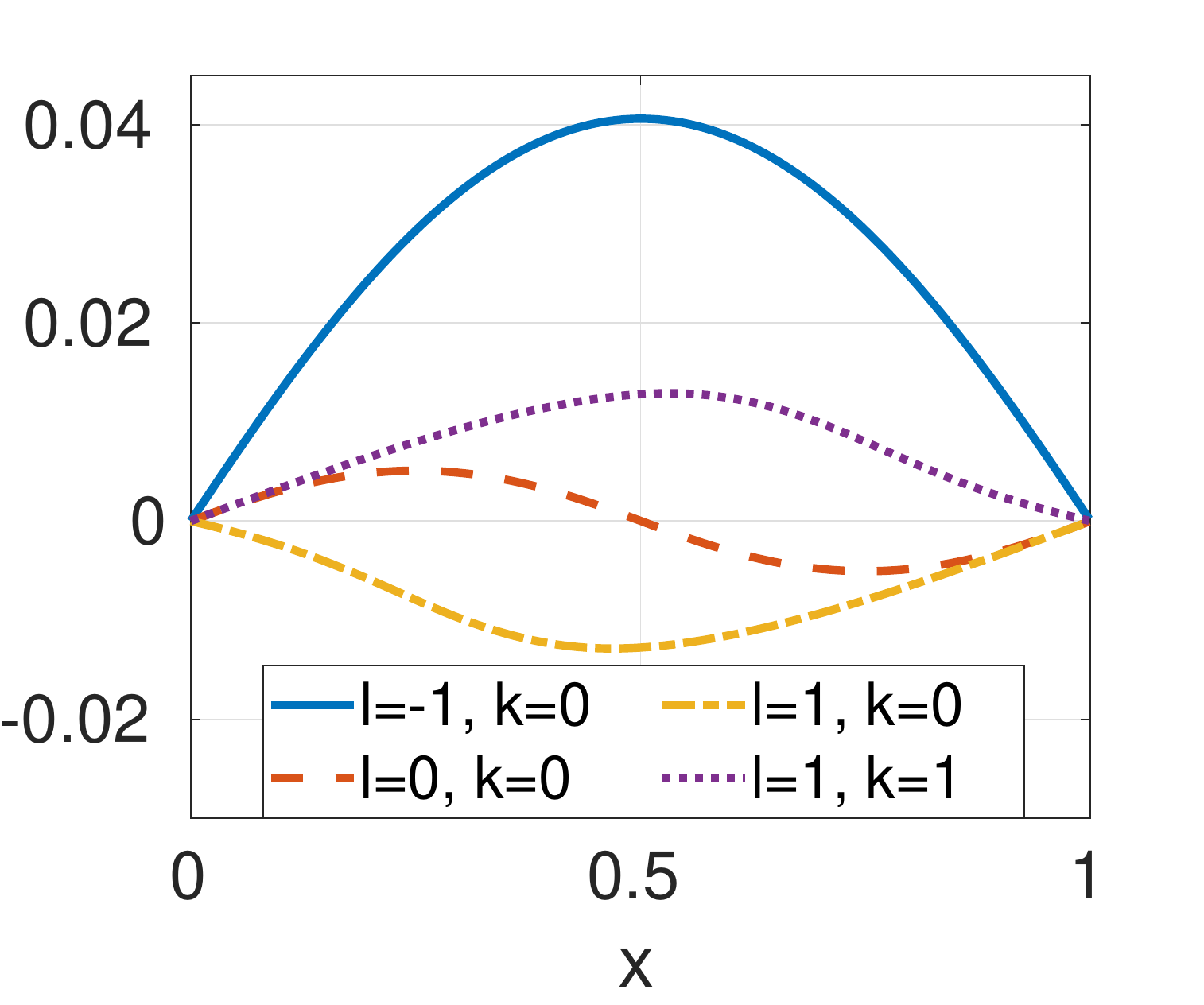}
\hfill
\includegraphics[width=.3\textwidth]{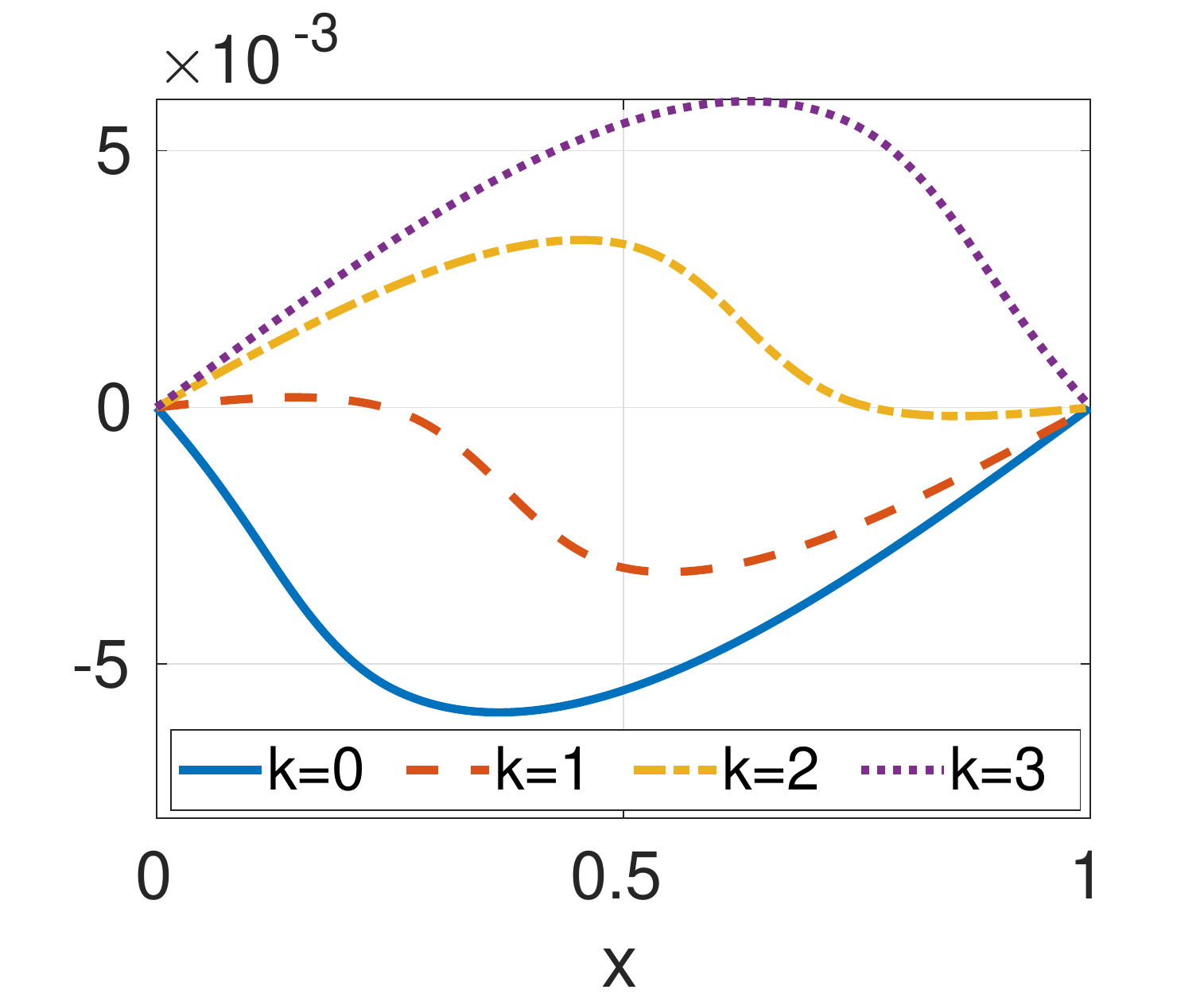}
\hfill
\includegraphics[width=.3\textwidth]{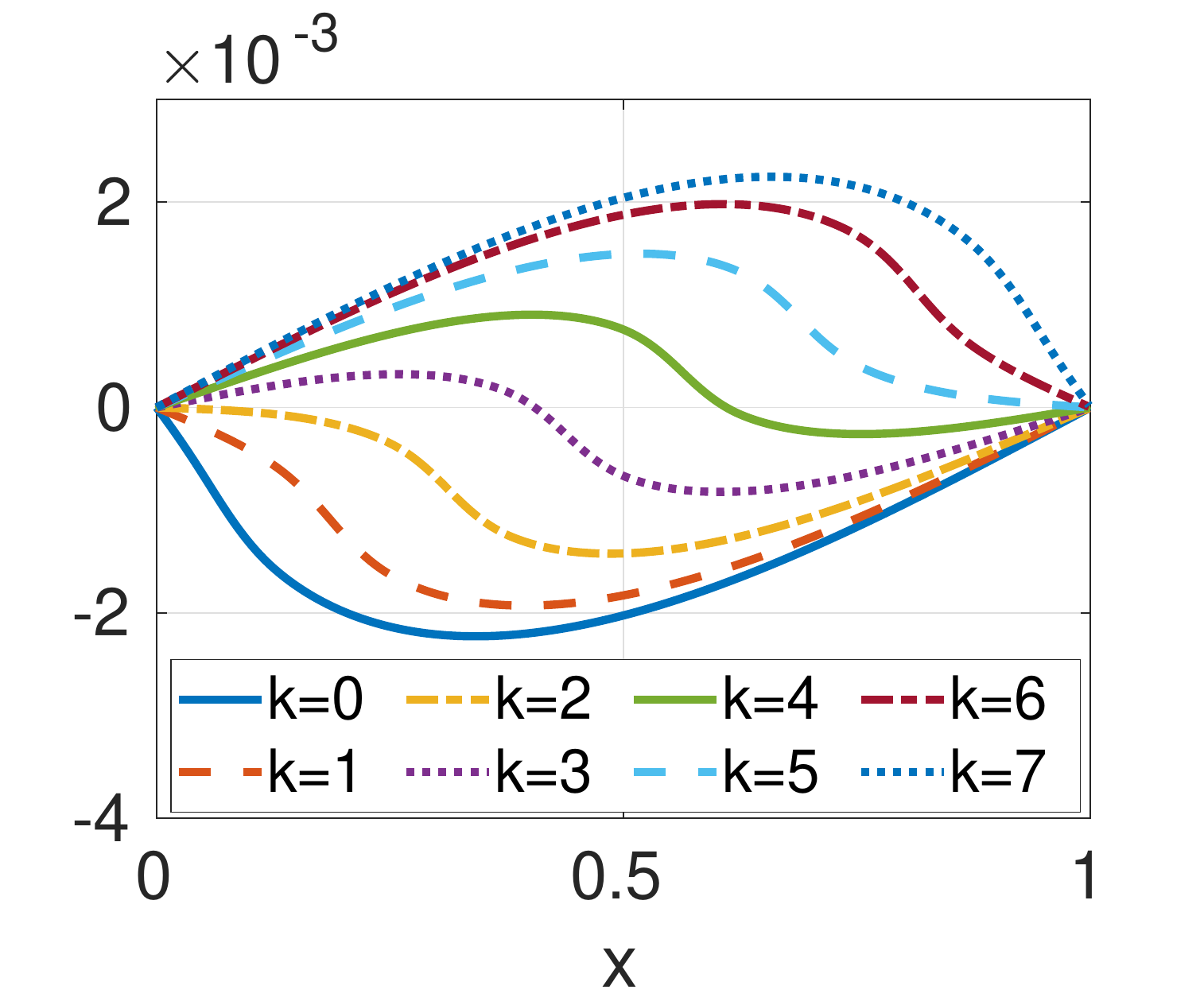}
\caption{Expansion functions resulting from applying $\CovOp^{1/2}$ as in Example \ref{exam:smoothBB} for $q=3$ to the Haar wavelets $\psi_{m}$, $m= 2^\ell+k$, with level $\ell \in \{-1,0,1\}$ (left), $\ell=2$ (middle), and $\ell = 3$ (right).}
\label{fig:Phi_Wavelet}
\end{figure}

\section{Sparse Grid Approximation} \label{sec:sparse_grids}

In \cite{ErnstEtAl2018} we presented a solution approach for solving random elliptic PDEs based on sparse polynomial collocation derived from tensorized interpolation at Gauss-Hermite nodes.
The problem is cast as that of approximating the solution $u$ of \eqref{equ:PDE} as a function $u:\Gamma \to H_0^1(D)$ by solving for realizations of $u$ associated with judiciously chosen \emph{collocation points} $\{\vxi_j\}_{j=1}^N \subset \Gamma$.

Sparse polynomial collocation operators are constructed by tensorizing univariate Lagrange interpolation sequences $(U_k)_{k \in \mathbb N_0}$ defined as
\begin{equation}\label{eq:U_k-def}
	(U_k f)(\xi) = \sum_{i=0}^k f(\xi_{i}^{(k)}) \, L_i^{(k)}(\xi), 
	\qquad 
	f\colon \bbR \to \bbR, 
\end{equation}
where $\{L_i^{(k)}\}_{i=0}^k$ denote the Lagrange fundamental polynomials of degree $k$ associated with a set of $k+1$ distinct interpolation nodes $\Xi^{(k)} := \bigl\{ \xi_{0}^{(k)},\xi_{1}^{(k)},\dots,\xi_{k}^{(k)} \bigr\} \subset \bbR$ and $L_0 \equiv 1$.
For any $\vk \in \mc F$ (cf.\ \eqref{multi-indices}),  the associated tensorized Lagrange interpolation operator $U_\vk := \bigotimes_{m \in \bbN} U_{k_m}$ is given by 
\begin{equation} \label{eq:tensor-int}
   (U_\vk f)(\vxi)
   =
   \left(\bigotimes_{m \in \bbN} U_{k_m} f \right)(\vxi)
   =
   \sum_{\vi \le \vk} f(\vxi_{\vi}^{(\vk)}) L_\vi^{(\vk)}(\vxi),
  \qquad
  f:\bbR^\bbN \to \bbR,
\end{equation}
in terms of the tensorized Lagrange 
polynomials 
$L_\vi^{(\vk)}(\vxi) := \prod_{m \in \mathbb N} L_{i_m}^{(k_m)}(\xi_m)$ with multivariate interpolation nodes
$\vxi_\vi^{(\vk)}  \in     \Xi^{(\vk)} := \bigtimes_{m \in \bbN} \Xi^{(k_m)}$.
We thus have $U_\vk: \bbR^\Gamma \to \mc Q_\vk$, where
\[
    \mc Q_\vk := \Span \{ \vxi^\vi : 0 \le i_m \le k_m, m \in \bbN \},
    \qquad \vk \in \mc F,
\]
denotes the multivariate tensor-product polynomial space of maximal degree $k_m$ in the $m$-th variable in the countable set of variables $\vxi = (\xi_m)\in \bbR^\bbN$.

Sparse polynomial spaces can be constructed by tensorizing the univariate \emph{detail operators} 
\begin{equation} \label{equ:detail}
	\Delta_k := U_k - U_{k-1}, \quad k\geq 0,
	\qquad U_{-1} :\equiv 0,
\end{equation}
giving
\[
	\Delta_{\vk} := \bigotimes_{m \in \bbN} \Delta_{k_m} \colon \bbR^\Gamma \to \mc Q_{\vk}.
\]
A sparse polynomial collocation operator is then obtained by fixing a suitable set of multi-indices $\Lambda \subset \mc F$ and setting
\begin{equation} \label{equ:U_Lambda}
	U_{\Lambda} := \sum_{\vi \in \Lambda} \Delta_{\vi} : \bbR^\Gamma \to \mc P_{\Lambda},
	\qquad \text{ where }
	\mc P_{\Lambda} := \sum_{\vi \in \Lambda} \mc Q_\vi.
\end{equation}
It is shown in \cite{ErnstEtAl2018} that if $\Lambda$ is finite and \emph{monotone} (meaning that $\vi \in \Lambda$ implies that any $\vj \in \mc F$ for which $\vj \le \vi$ holds componentwise also belongs to $\Lambda$), 
then $U_\Lambda$ is the identity on $\mc P_\Lambda$ and $\Delta_\vi$ vanishes on $\mc P_\Lambda$ for any $\vi \not\in \Lambda$.

The construction of $U_\Lambda f$ for $f\colon \Gamma \to \bbR$ consists of a linear combination of tensor product interpolation operators requiring the evaluation of $f$ at certain multivariate nodes.
It can be shown that for $\vi \in \mc F$ the detail operators have the representation
\[
	\Delta_\vi f 
	= \Big[ \bigotimes_{m\geq 1} (U_{i_m} - U_{i_m-1})  \Big]f
	= \sum_{\vi - \boldsymbol 1 \leq \vk \leq \vi} (-1)^{|\vi-\vk|_1} \Big[\bigotimes_{m\geq 1} U_{k_m}\Big]f,
\]
leading to an alternative representation of $U_\Lambda$ for monotone finite subsets $\Lambda \subset \mc F$ known as the \emph{combination technique}:
\begin{equation}\label{equ:combitec}
	U_\Lambda = \sum_{\vi \in \Lambda} c(\vi;\Lambda)\, U_\vi,
	\qquad
	c(\vi;\Lambda) := \sum_{\ve \in \{0,1\}^\bbN\colon \vi+\ve \in \Lambda} (-1)^{|\ve|_1}.
\end{equation}
We refer to the collection of nodes appearing in the tensor product interpolants $U_{\vi}$ as the \emph{sparse grid} $\Xi_\Lambda \subset \Gamma$ associated with $\Lambda$:
\begin{equation}\label{equ:sparse_grid}
	\Xi_\Lambda = \bigcup_{\vi \in \Lambda} \Xi^{(\vi)}.
\end{equation}
In the same way, when approximating the solution $u:\Gamma \to H_0^1(D)$ of \eqref{equ:PDE} by $u(\vxi) \approx (U_\Lambda u)(\vxi)$, each evaluation $u(\vxi_j)$ at a sparse grid point $\vxi_j \in \Xi_\Lambda$ represents the solution of the PDE for the coefficient $a = a(\vxi_j)$.

\begin{remark}\label{rem:on-sparse-grid-constr}
\rev{Let us provide some further comments.}
\begin{enumerate}
\item
The univariate interpolation operators $U_k$ in \eqref{eq:U_k-def}, on which the sparse \rev{grid} collocation construction is based, will have degree of exactness $k$, as the associated sets of interpolation nodes $\Xi^{(k)}$ have cardinality $k+1$.
Although we do not consider this here, allowing nodal sets to grow faster than this may bring some advantages. 
Such an example is the sequence of \emph{Clenshaw--Curtis nodes} (cf.\ \cite{NobileEtAl2008a}), for which $|\Xi^{(0)}| = 1$ and $|\Xi^{(k)}| = 1+2^{k}$.
\item
The Clenshaw--Curtis doubling scheme \rev{generates} \emph{nested} node sets $\Xi^{(k+1)} \subset \Xi^{(k)}$.
This has the advantage that higher order collocation approximations may re-use function evaluations of previously computed lower-order approximations.
Moreover, it was shown in \cite{BarthelmannEtAl2000} that sparse \rev{grid} collocation based on nested node sequences are interpolatory.
By contrast, the sequence of Gauss--Hermite nodes with $|\Xi^{(k)}|=k+1$ results in disjoint consecutive nodal sets.
The number of new nodes added by each consecutive set is referred to as the \emph{granularity} of the node sequence.

\item
Two heuristic approaches for constructing monotone multi-index sets $\Lambda$ for sparse polynomial collocation for lognormal random diffusion equations are presented in \cite{ErnstEtAl2018}. 
Further details are given in Section~4.
\end{enumerate}
\end{remark}

In \cite{ErnstEtAl2018}, a convergence theory for sparse polynomial collocation approximations $f \approx U_\Lambda f$ of functions in $f \in L^2_\mu(\Gamma, H_0^1(D))$ was given based on the expansion 
\[
   f(\vxi) = \sum_{\vk \in \mc F} f_\vk \, H_\vk(\vxi), 
   \qquad 
   f_\vk = \int_\Gamma f(\vxi) H_\vk(\vxi) \mu(\d \vxi),
\]
in tensorized Hermite polynomials $H_\vk(\vxi) = \prod_{m \in \mathbb N} H_{k_m}(\xi_{m})$, $\vk \in \mc F$, with $H_{k_m}$ denoting the univariate Hermite orthogonal polynomial of degree $k_m$, which are known to form an orthonormal basis of $L^2_\mu(\Gamma; H_0^1(D))$.

Under assumptions to be detailed below, it was shown \cite[Theorem 3.12]{ErnstEtAl2018} that there exists a nested sequence of monotone multi-index sets $\Lambda_N \subset \mc F$, where $N=|\Lambda_N|$, such that the sparse \rev{grid} collocation error of the approximation $U_{\Lambda_N}f$ satisfies
\begin{equation} \label{equ:coll-rate}
   	\left\| f - U_{\Lambda_N} f \right\|_{L^2_\mu} 
	\leq 
	C (1+N)^{-\left( \frac 1p - \frac 12\right )},
\end{equation}
for certain values of $p \in (0,2)$ with a constant $C$. 
The precise assumptions under which \eqref{equ:coll-rate} was shown to hold are as follows:
\begin{enumerate}[(1)]
\item
The condition $\mu(\Gamma)=1$ on the domain of $f$ (cf.\ \eqref{equ:Gamma}).
\item
An assumption of weighted $L^2_\mu$-summability on the derivatives of $f$: specifically, there exists $r \in \mathbb N_0$ \rev{such that $\partial^{\vk} f \in L^2_\mu(\mathbb R^{\mathbb N}; H_0^1(D))$ for all $\vk \in \mc F$ with $|\vk|_\infty \le r$ and a sequence of positive numbers $(\tau_m^{-1})_{m \in \mathbb N} \in \ell^p(\mathbb N)$, $p \in (0,2)$, such that} relation \eqref{equ:bound_partial_weighted} holds.
\item
An assumption on the univariate sequence of interpolation nodes: there exist constants $\theta \ge 0$ and $c \ge 1$ such that the univariate detail operators \eqref{equ:detail} satisfy
\begin{equation} \label{equ:Delta_bound}
 	 \max_{i \in \bbN_0} \|\Delta_i H_k\|_{L^2_\mu} \leq (1+ c k)^\theta,
	\qquad k \in\bbN_0.
\end{equation}
\end{enumerate}
In order \rev{for \eqref{equ:coll-rate} to hold true}, it is sufficient that \eqref{equ:bound_partial_weighted} be satisfied for $r  > 2(\theta+1) + \frac 2p$.
It was shown in \cite[Lemma 3.13]{ErnstEtAl2018} that \eqref{equ:Delta_bound} holds with $\theta = 1$ for the detail operators $\Delta_k = U_k - U_{k-1}$ associated with univariate Lagrange interpolation operators $U_k$ at Gauss-Hermite nodes, i.e., the zeros of the univariate Hermite polynomial of degree $k+1$.

\subsection{Gaussian Leja Nodes}
\label{sec:Gaussian_Leja}
\emph{Leja points} for interpolation on a bounded interval $I \subset \bbR$ are defined recursively by fixing an arbitrary initial point $\xi_0 \in I$ and setting
\begin{equation} \label{equ:leja-def}
	\xi_{k+1} := \argmax_{\xi \in I} \prod_{i=1}^k |\xi - \xi_i|,
	\qquad
	k \in \mathbb N_0.
\end{equation}
They are seen to be nested, possessing the lowest possible granularity and have been shown to have an asymptotically optimal distribution \cite[Chapter~5]{SaffTotik1997}.
The quantity maximized in the extremal problem \eqref{equ:leja-def} is not finite for unbounded sets $I$, which arise, e.g., when an interpolation problem is posed on the entire real line. 
Such is the case with parameter variables $\xi_m$ which follow a Gaussian distribution. 
By adding a weight function vanishing at infinity faster than polynomials grow, one can generalize the Leja construction to unbounded domains (cf.\ \cite{Lubinsky2007}). 
Different ways of incorporating weights in \eqref{equ:leja-def} have also been proposed in the bounded case, cf.\ e.g.\ 
\cite[p.\ 258]{SaffTotik1997},
\cite{BaglamaEtAl1998}, and
\cite{DeMarchi2004}.
In \cite{NarayanJakeman2014}, it was shown that for weighted Leja sequences generated on unbounded intervals  $I$ by solving the extremal problem
\begin{equation} \label{equ:weighted Leja}
   \xi_{k+1} = \argmax_{\xi \in I}  \sqrt{\rho(\xi)} \prod_{i=0}^{k} |\xi - \xi_i|,
\end{equation}
where $\rho$ is a probability density function on $I$, their asymptotic distribution coincides with the probability distribution associated with $\rho$.
This is shown in \cite{NarayanJakeman2014} for the generalized Hermite, generalized Laguerre and Jacobi weights, corresponding to a generalized Gaussian, Gamma and Beta distributions.
Subsequently, the result of \cite{TaylorTotik2010} on the subexponential growth of the Lebesgue constant of bounded unweighted Leja sequences was generalized to the unbounded weighted case in \cite{JantschEtAl2019}.

If we choose $\rho(\xi) = \exp(-\xi^2/2)$ and $I = \bbR$ in \eqref{equ:weighted Leja} and set $\xi_0 = 0$, then we shall refer to the resulting weighted Leja nodes also \emph{Gaussian Leja nodes} in view of their asymptotic distribution.
Unfortunately, the result in \cite{JantschEtAl2019} does not imply a bound like \eqref{equ:Delta_bound} for univariate interpolation using Gaussian Leja nodes.
However, we provide numerical evidence in Figure \ref{fig:Delta_Norms} suggesting that \eqref{equ:Delta_bound} is also satisfied for Gaussian Leja nodes with $\theta = 1$.
\begin{figure}
\begin{minipage}[c]{0.5\textwidth}
\centering
\includegraphics[width=\textwidth]{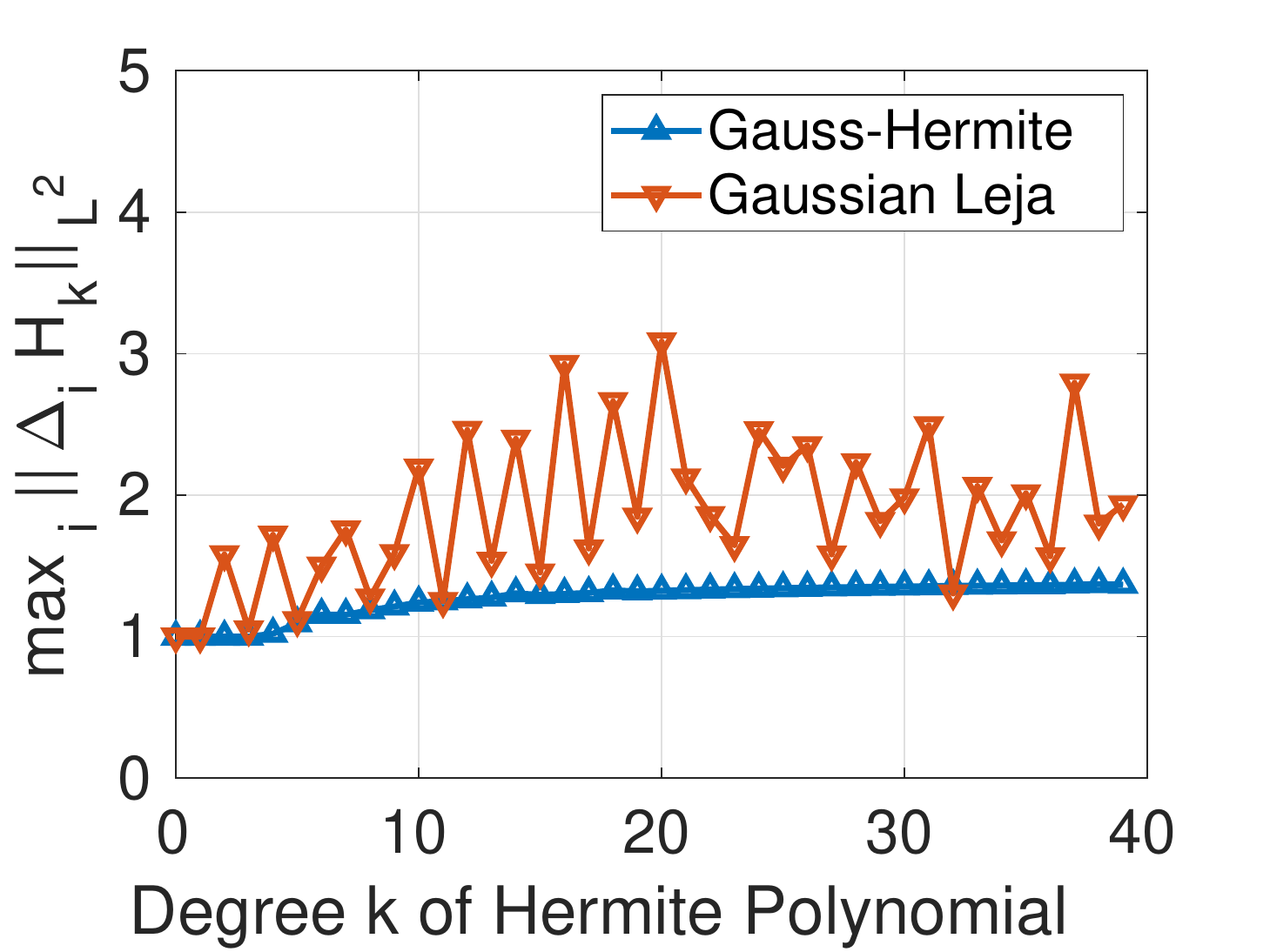}
\end{minipage}
\begin{minipage}[b]{0.49\textwidth}
\caption{Comparison of $\max_{i} \|\Delta_i H_k\|_{L^2_\mu}$, $k = 1,\ldots,39$, for Gauss--Hermite and Gaussian Leja nodes.}
\label{fig:Delta_Norms}
\vfill\phantom{a}
\end{minipage}
\end{figure}
In the next subsection we compare the performance of Gaussian Leja nodes for quadrature and interpolation purposes to that of Gauss--Hermite and Genz--Keister nodes \cite{GenzKeister1996}, which represent another common univariate node family for quadrature w.r.t.~a Gaussian weight.
Although a comparison of Gaussian Leja with Genz--Keister points is already available in \cite{NarayanJakeman2014} and a comparison between Gauss--Hermite and Genz--Keister points is reported in \cite{NobileEtAl2016,Chen2016}, the joint comparison of the three choices has not been reported in literature to the best of our knowledge.

\subsection{Performance Comparison of Common Univariate Nodes} 
\label{sec:Comp_Nodes}
In this section we investigate \rev{and compare the performance of} numerical quadrature and interpolation of uni- and \rev{multi}variate
functions \rev{($M=2,6,9$ variables)} using either Gauss--Hermite, Genz--Keister or Gaussian Leja nodes.
\rev{As a measure of performance we consider the achieved error in relation to the number of employed quadrature or interpolation nodes, respectively. 
Quadrature is carried out with respect} to a standard (multivariate) Gaussian measure $\mu$ and the interpolation error is measured in $L^2_\mu$.
The functions we consider in this section were previously proposed in \cite{Tamellini2012} for the purpose of comparing univariate quadrature with Gauss--Hermite and Genz--Keister points and are \rev{included} in the figures displaying the results.

Quadrature results are reported in Figure \rev{\ref{fig:quad_results}}. 
In the univariate case, Gauss--Hermite nodes \rev{perform best}, 
and Genz--Keister nodes also show good performance, which is not surprising given that they are constructed as nested extensions of the Gauss--Hermite points with maximal degree of exactness.
The Gaussian Leja nodes, by comparison, perform poorly.
This should not surprise, however, given that Gaussian Leja points are determined by minimizing
Lebesgue constants, i.e., they are conceived as interpolation points rather than quadrature points.

In the \rev{multi}variate case, however, the situation changes and Gauss--Hermite nodes are the worst performing\rev{.
This is} due to their non-nestedness, which tends to introduce unnecessary quadrature nodes into the quadrature scheme.
Note that in this case we are simply using the standard Smolyak sparse multi-index set in $M$ dimensions in Equation \eqref{equ:U_Lambda},
\[
	\Lambda_w = \biggl\{ \vi \in \bbN^M : \sum_{m=1}^M i_m \leq w \biggr\},  
	\quad \mbox{for some } w \in \bbN,
\]
i.e., we are not tailoring the sparse grid either to the function to be integrated nor to the univariate points.
\rev{The Gaussian Leja and Genz--Keister points show a faster decay of the quadrature error, due to their nestedness.
  This is remarkable in particular for Gaussian--Leja, given that they were proposed in literature as univariate interpolation points, as already discussed.
  Overall, the Genz--Keister points show the best performance as expected, but it is important to recall that
  only a limited number of Genz--Keister nodes is available, i.e., no nested Genz--Keister quadrature formula with real
  quadrature weights and more than 35 nodes is known in literature, \cite{GenzKeister1996,Tamellini2012,heiss.winschel:kpnquad}. 
  In particular, the plots report the largest standard sparse grids that can be built with these rules before running out of tabulated Genz--Keister points.}

\rev{We remark that introducing a Genz--Keister quadrature formula with more than 35 nodes
  is not a simple matter of investing more computational effort and tabulating more points, but it would
  entail some ``trial and error'' phase to look for a suitable sequence of so-called ``generators'', see e.g. \cite{Tamellini2012} for more details.
  This activity exceeds the scope of this paper.
  Moreover, Genz--Keister nodes are significantly less granular, which could be a disadvantage in certain situations:
  indeed, the cardinalities of the univariate Genz--Keister node sets are $|\Xi^{(k)}| = 1, 3, 9, 19, 35$ for $k=0,\dots,4$
  (and a sequence of Genz--Keister sets exceeding $35$ nodes might be even less granular, e.g., jumping from 1 to 5 or 7).}
\begin{figure}[t]
  \centering
  \includegraphics[width=\linewidth]{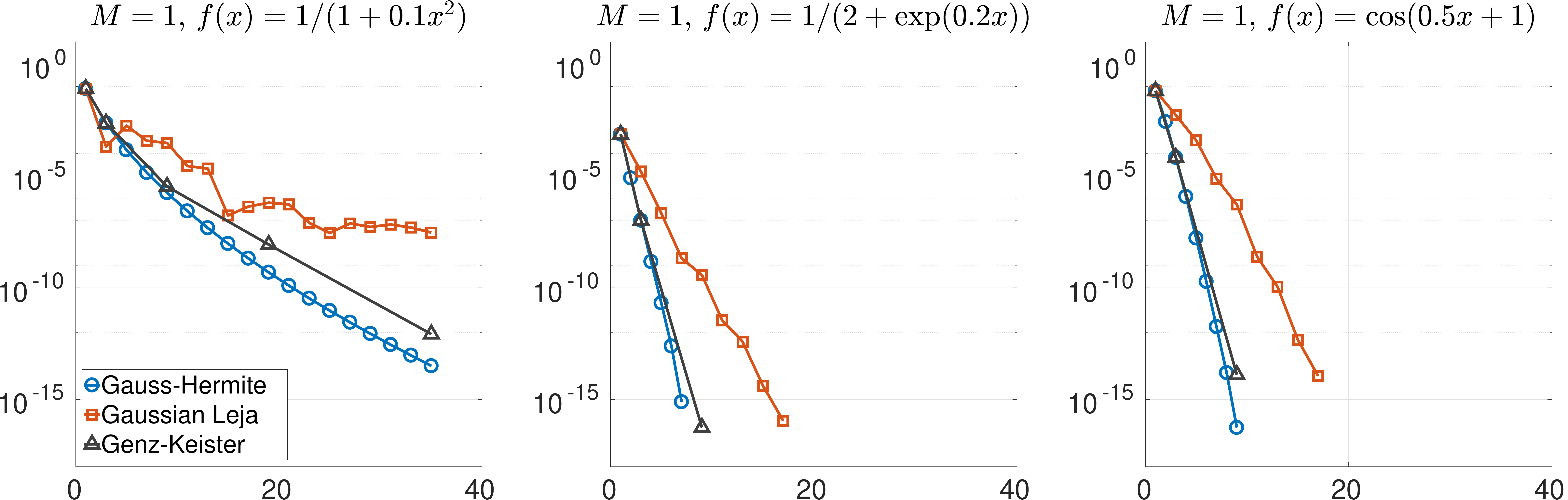}
  \includegraphics[width=\linewidth]{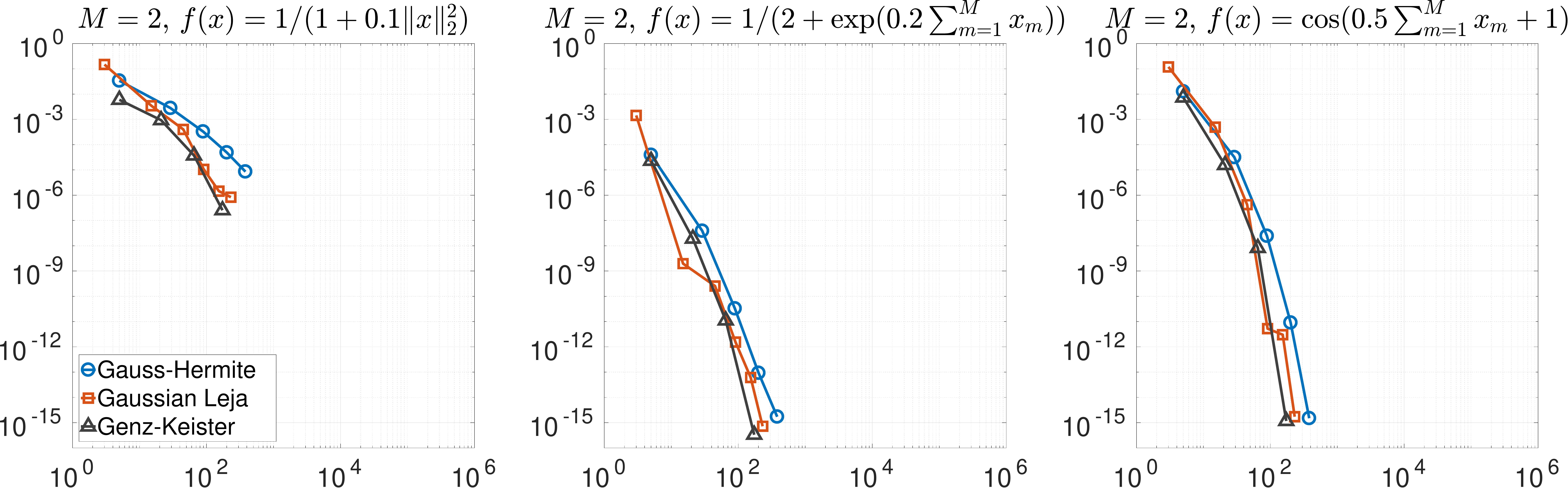}
  \includegraphics[width=\linewidth]{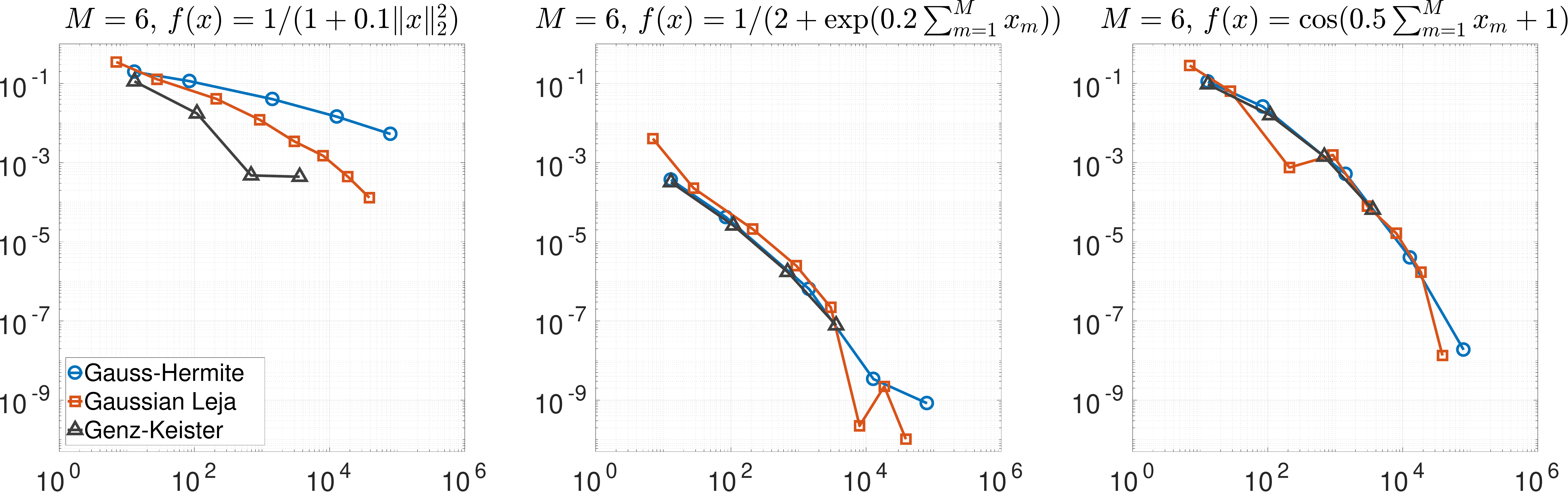}
  \includegraphics[width=\linewidth]{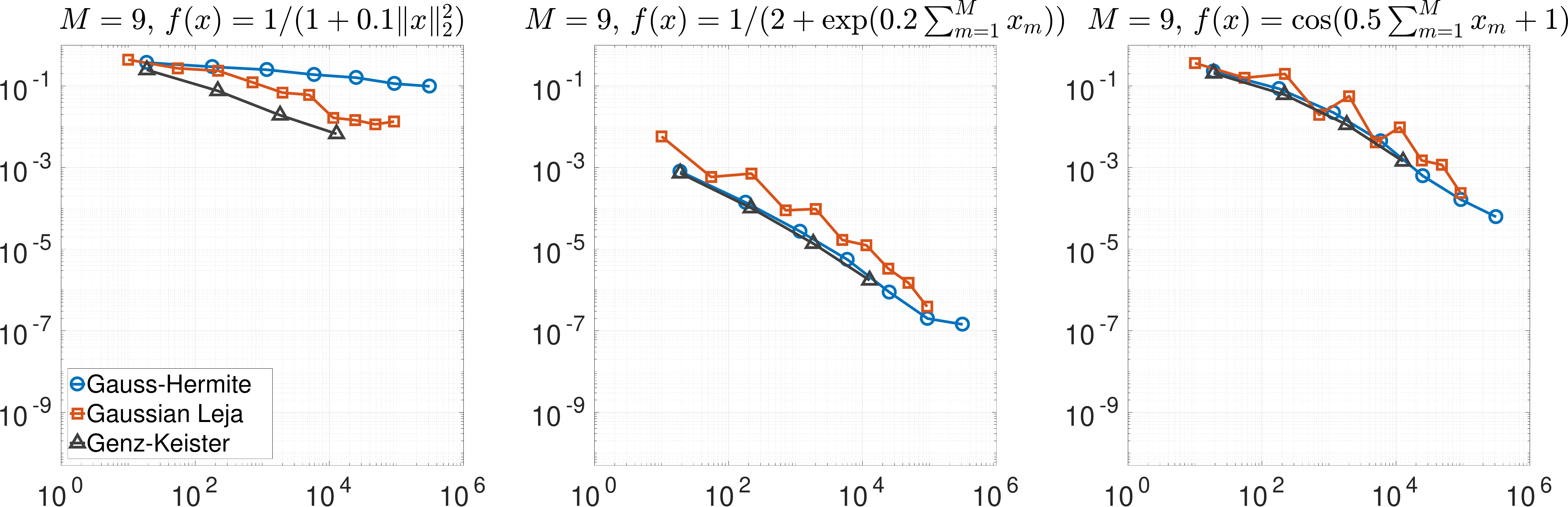}
  \caption{Results for univariate \rev{and multivariate} quadrature test.}
  \label{fig:quad_results}
\end{figure}
Next, \rev{we turn to comparing the performance of the different node families for interpolation.
Here,} Gaussian Leja nodes are expected to be best (or close-to-best) performing, given their specific design. 
Measuring interpolation error on unbounded domains with a Gaussian measure (or any non-uniform measure for that matter)
is a delicate task, as one would need to choose a proper weight to ensure boundedness
of the pointwise error, see e.g. \cite{Harbrecht2016,NobileEtAl2016}.
In this contribution, we actually discuss the $L^2_\mu$ approximation error of the interpolant, which we compute as follows: we sample $K$ independent batches of $M$-variate Gaussian random variables, with $P$ points each, $\mathcal{B}_k = \{\vxi_{i}\}_{i=1}^P, \xi_{i,m} \sim \mathsf{N}(0,1), m=1,\ldots,M, k=1,\ldots,K$; we construct a sequence of increasingly accurate sparse grids $U_{\Lambda_w}[f]$ and evaluate them on each random batch; we then approximate the $L^2_\mu$ error for each sparse grid on each batch by Monte Carlo,
\[
    \text{Err}_{k}(U_{\Lambda_w}[f]) 
    = 
    \frac{1}{P} \sum_{i=1}^P ( f(\vxi_i) - U_{\Lambda_w}[f](\vxi_i))^2
\]
and then we show the convergence of the median value of the $L^2_\mu$ error for each sparse grid over the $K$ repetitions.%
\footnote{Exchanging the median value with the mean value does not significantly change the plots, which means that the errors are distributed symmetrically around the median. 
For brevity, we do not report these plots here.
We have also checked that the distribution of the errors is not too spread, by adding boxplots to the convergence lines. 
Again, we do not show these plots for brevity. 
Finally, observe that we could have also employed a sparse grid to compute the $L^2_\mu$ error, but we chose Monte Carlo quadrature to minimize the chance that the result depends on the specific grid employed.}
\rev{The results are reported in Figure \ref{fig:interp_results}. The plots indicate that the convergence of interpolation
  degrades significantly as the number of dimension $M$ increases (due to the simple choice of index-set $\Lambda_w$),
  and in particular the convergence of grids based on Gauss--Hermite points is always the worst among those tested
  (due again to their non-nestedness),
  so that using nested points such as Gaussian Leja or Genz--Keister becomes mandatory. The performance of Genz--Keister points is
  surprisingly good, even better than Gaussian Leja at times, despite the fact that they are designed for quadrature rather than interpolation.
  However, the rapid growth 
and the limited availability of Genz--Keister points still are substantial drawbacks.
  To this end, we remark that also in these plots we are showing the largest grid that we could compute before running out of Genz--Keister points.} 

\begin{figure}[t]
  \centering
  \includegraphics[width=\linewidth]{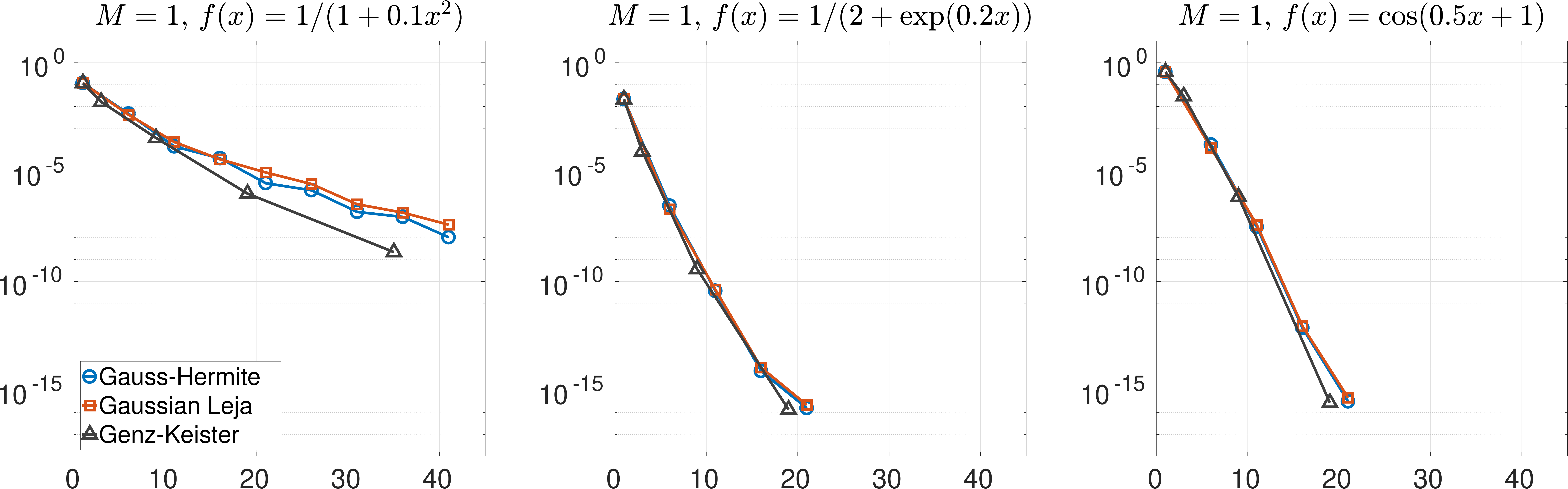}
  \includegraphics[width=\linewidth]{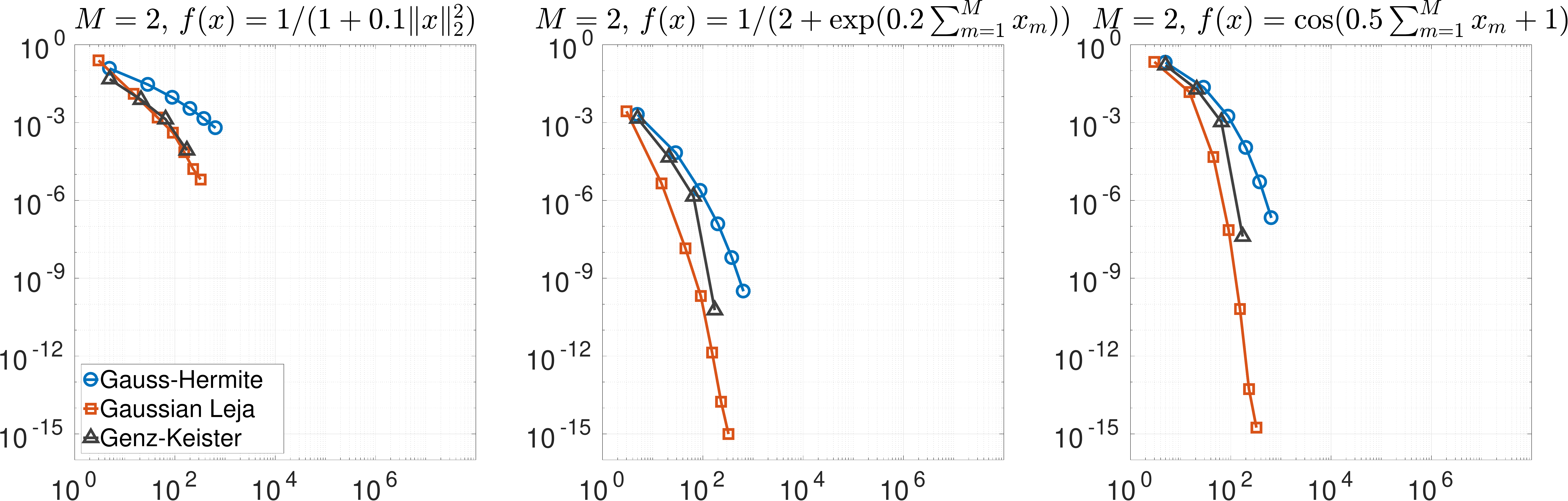}
  \includegraphics[width=\linewidth]{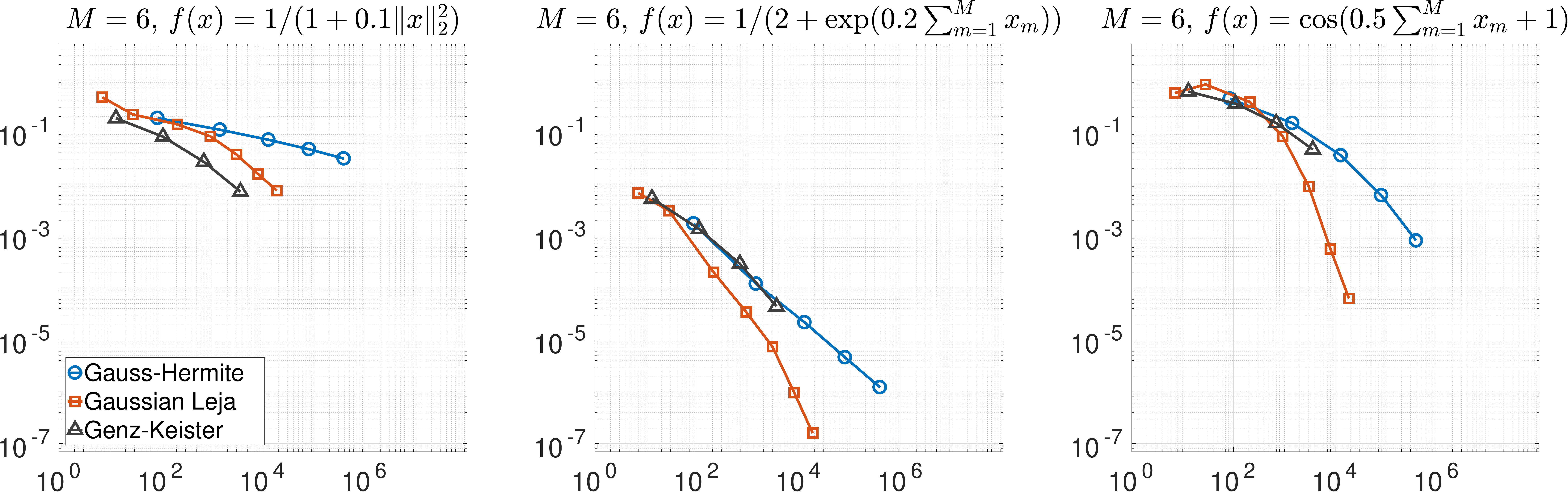}
  \includegraphics[width=\linewidth]{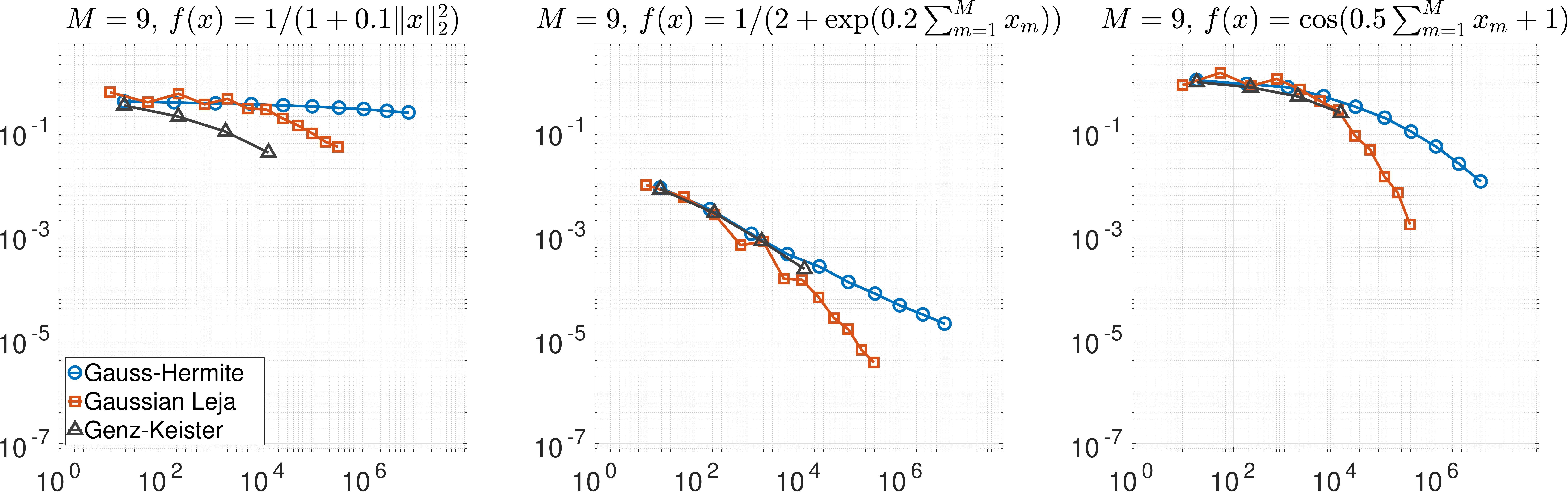}
  \caption{$L^2_\mu$ error for \rev{univariate and }multivariate interpolation.
    \rev{The results for the univariate test were produced with $K=30$ repetitions, each with $P=100$ samples}.
    The results \rev{for the multivariate test} were produced with $K=50$ repetitions, each with $P=500$ samples.}
  \label{fig:interp_results}
\end{figure}

\section{Numerical Results} \label{sec:numerics}
We now perform numerical tests solving the elliptic PDE introduced in Section \ref{sec:pde},
with the aim of extending the numerical evidence obtained in \cite{ErnstEtAl2018}.
In that paper, we assessed:
\begin{itemize}
\item 
the sharpness of the predicted rate for the a-priori sparse grid construction (both with respect to the number of \rev{multi-}indices in the set and the number of points in the sparse grids);
\item 
the comparison in performance of the a-priori and the classical dimension-adaptive a-posteriori sparse grid constructions;
\end{itemize}
limiting ourselves to Gauss--Hermite collocation points, which are covered by our theory. 
The findings indicated that our predicted rates are 
\rev{somewhat conservative. Specifically, }the rates of convergence measured in numerical experiments were larger than the theoretical ones by a factor between 0.5 and 1, cf. \cite[Table 1]{ErnstEtAl2018}.
This is due to some technical estimates applied in the proof of the convergence results which \rev{we were so far not able to improve}.
Concerning the second \rev{point, we observed in \cite{ErnstEtAl2018} that} the a-priori construction is actually competitive with the a-posteriori adaptive variant, especially if one considers the extra PDE solves needed to explore the set of multi-indices. 

We remark in particular that we observed convergence of the sparse grid approximations even in cases 
in which the theory predicted no convergence (albeit with a rather poor convergence rate, comparable to that attainable with Monte Carlo or Quasi Monte Carlo methods---see also \cite{NobileEtAl2016,tesei:MCCV} for possible remedies).

\rev{In this contribution, our goal is the numerical investigation of additional questions that so far remain unanswered by existing theory, among these:}
\begin{enumerate}
\item 
whether using Gaussian Leja or Genz--Keister nodes yields improvement over Gauss--Hermite nodes in our framework, see Section \ref{subs:GH-GK-LJ};
\item 
whether changing the random field representation from Karhunen-Lo\`eve (KL) to Lévy-Ciesielski (LC) expansion for the case $q=1$ (pure Brownian bridge) improves the efficiency of the numerical computations, see Section \ref{subs:KLEvsLCE}.
As explained above, this is motivated by the fact that LC expansion of the random field
allowed \cite{BachmayrEtAl2015} to prove convergence of the best-N-term approximation of the lognormal problem over Hermite polynomials.
\end{enumerate}

The tests were performed using the Sparse Grids Matlab Kit\footnote{v.18-10 ``Esperanza'',
which can be downloaded under the BSD2 license at \url{https://sites.google.com/view/sparse-grids-kit}.}.
We briefly recall the basic approaches of the two heuristics employed for constructing the multi-index sets $\Lambda$.
We refer to \cite{ErnstEtAl2018} for the full details of the two algorithms.
The first is the classical dimension-adaptive algorithm introduced by Gerstner and Griebel in \cite{GerstnerGriebel2003} with some suitable modifications to make it work with non-nested quadrature rules and for quadrature/interpolation on unbounded domains. 
It is driven by a posteriori error indicators computed along the outer margin of the current multi-index set.
The mechanism by which new random variables are activated in the multi-index set uses a ``buffer'' of fixed size containing variables whose error indicators have been computed but not yet selected. 
The second approach is an a-priori tailored choice of multi-index set $\Lambda$, which can be derived from the study of the decay of the spectral coefficients of the solution.

We thus consider the problem in \rev{Equation}~\eqref{equ:PDE} with $f=1$.
We set the pointwise standard deviation of $\log a$ \rev{to be} $\sigma=3$; note that this constant
does not appear explicitly in the expression for $\log a$ in Section \ref{sec:pde}, i.e., it has been absorbed in $\phi_m$.
Figure \ref{fig:a_realiz} shows 30 realizations of the random field $a(\omega)$ for different values of $q$, obtained by truncating the Karhunen-Lo\`eve expansion  of $a(\omega)$ at $M=1000$ random variables.
Specifically, we consider a smoothed Brownian bridge as in Example \ref{exam:smoothBB}, with
$q=3, \, 1.5, \, 1$, cf. \rev{Equation} \eqref{equ:BB_KLE_smooth}; for these values of $q$ a truncation at 1000 random variables covers $100\%$, $99.99996\%$ and $99.93\%$ of the total variance of $\log a$, respectively.
The plot shows how the realizations grow increasingly rough as $q$ decreases.
Upon plotting the corresponding PDE solutions (not displayed for brevity) one would \rev{observe that, by contrast,} solutions are much less rough, even in the case $q=1$.

\begin{figure}[t]
  \centering
  \includegraphics[width=\linewidth]{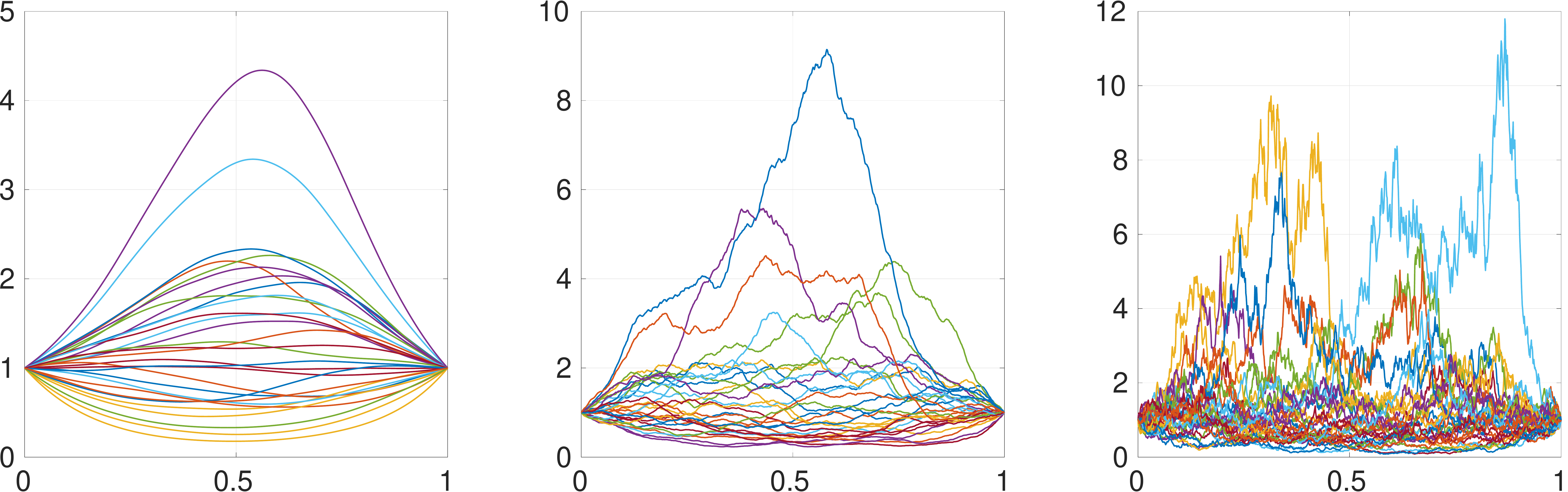}
  \caption{30 realizations of the random field for different values of $q$. Left: $q=3$; center: $q=1.5$; right: $q=1$. Note the different scaling of the vertical axis.}
  \label{fig:a_realiz}
\end{figure}

\subsection{Gauss--Hermite vs. Gaussian Leja vs. Genz--Keister nodes} \label{subs:GH-GK-LJ}

\begin{figure}[t]
  \centering
  \includegraphics[width=\linewidth]{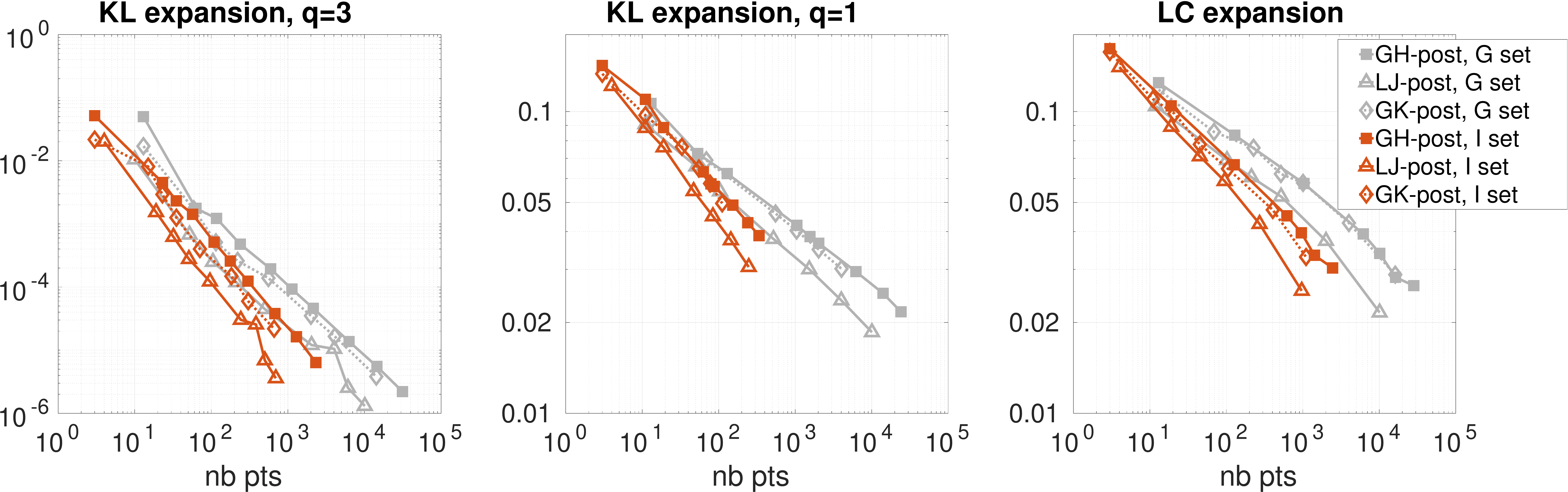}
  \includegraphics[width=\linewidth]{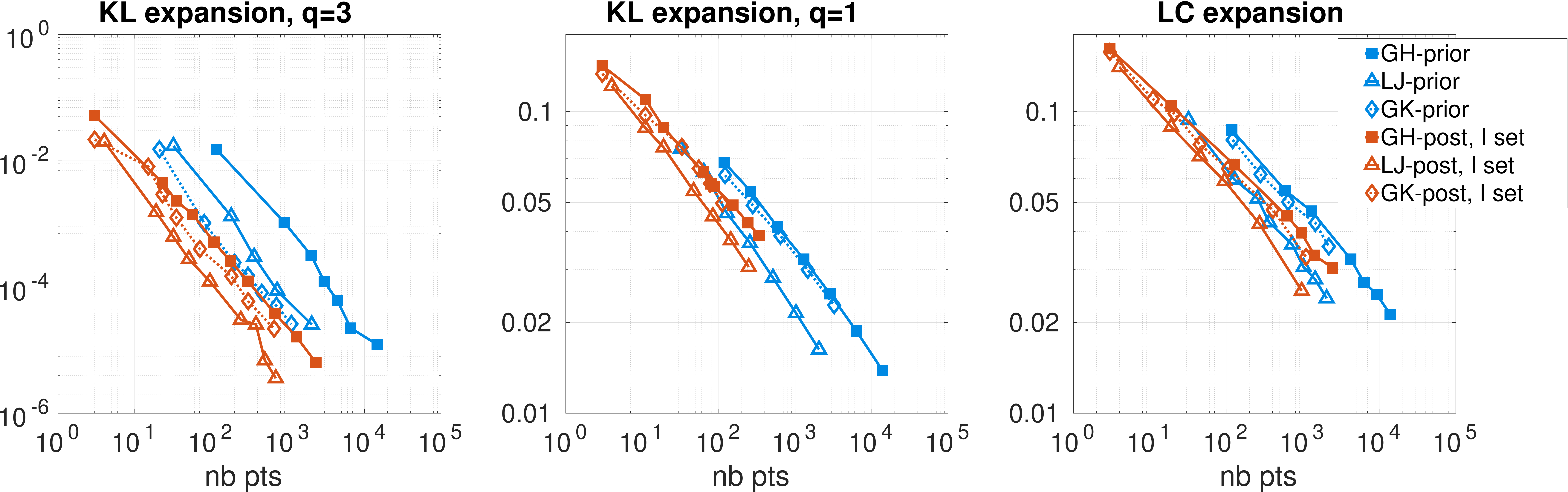}
  \includegraphics[width=\linewidth]{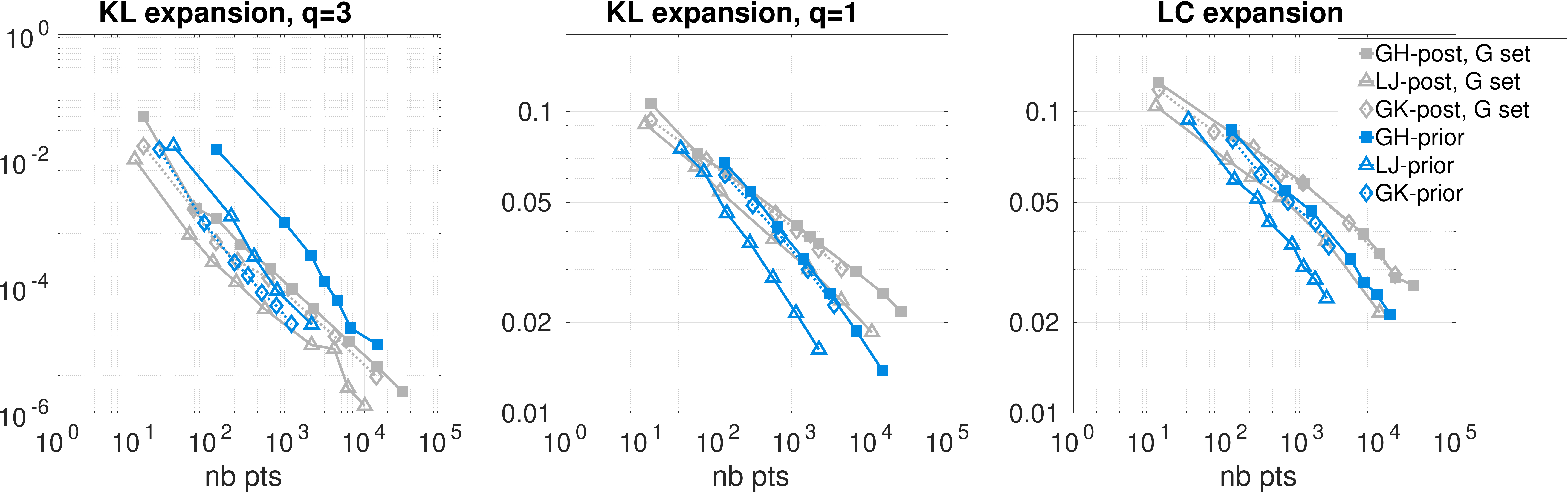}
  \caption{Comparison of performance for Gaussian Leja, Genz--Keister, and Gauss--Hermite points for different test cases and different
    sets of multi-indices. The plots report error versus number of points. \rev{To make the visual comparison easier
      we split the presentation into three parts.
The top row shows the two different sets produced by the a-posteriori algorithm
      (a-posteriori-I-set-incremental and a-posteriori-G-set-incremental).
The middle row compares a-priori and a-posteriori algorithms in terms of the optimal sets produced
      (a-priori-incremental and a-posteriori-I-set-incremental).
      The bottom row compares a-priori and a-posteriori algorithms in terms of bare computational cost
      (a-priori-incremental and a-posteriori-G-set-incremental).}}
  \label{fig:GHvsLJ}
\end{figure}
\begin{figure}[t]
  \includegraphics[width=\linewidth]{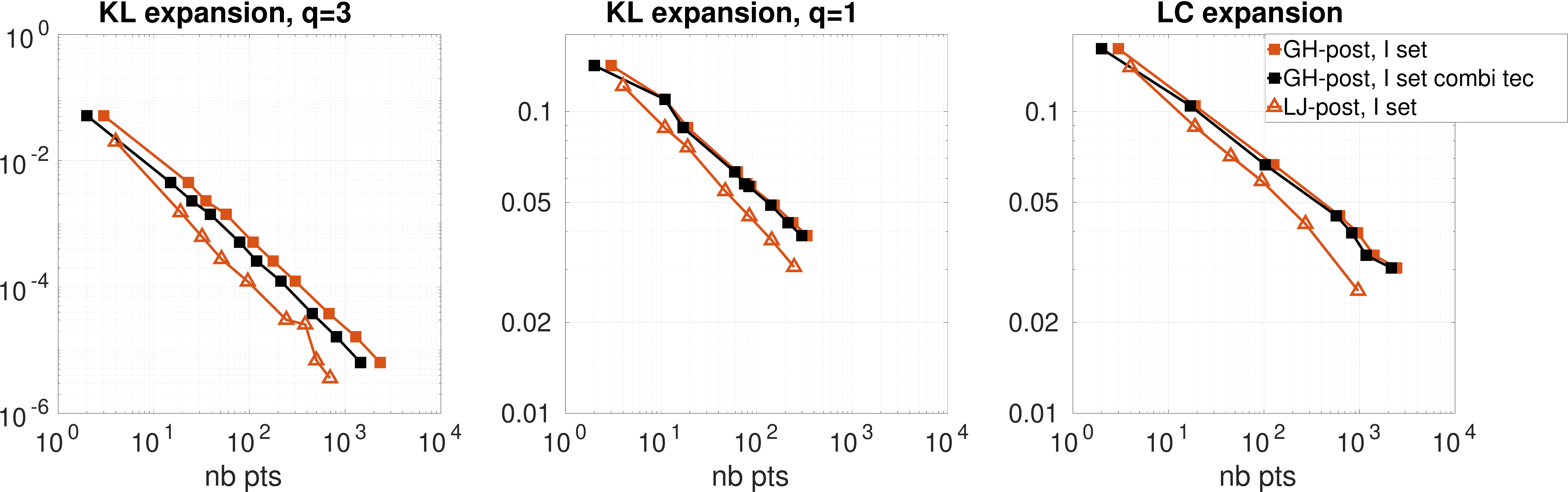}
  \includegraphics[width=\linewidth]{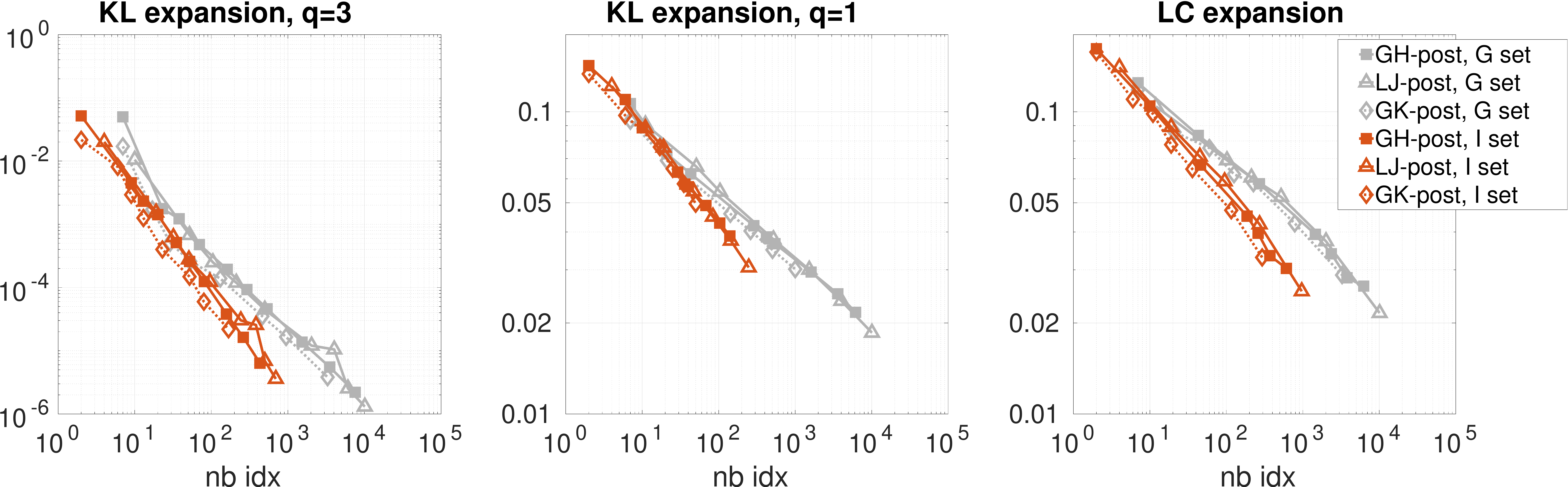}
  \includegraphics[width=\linewidth]{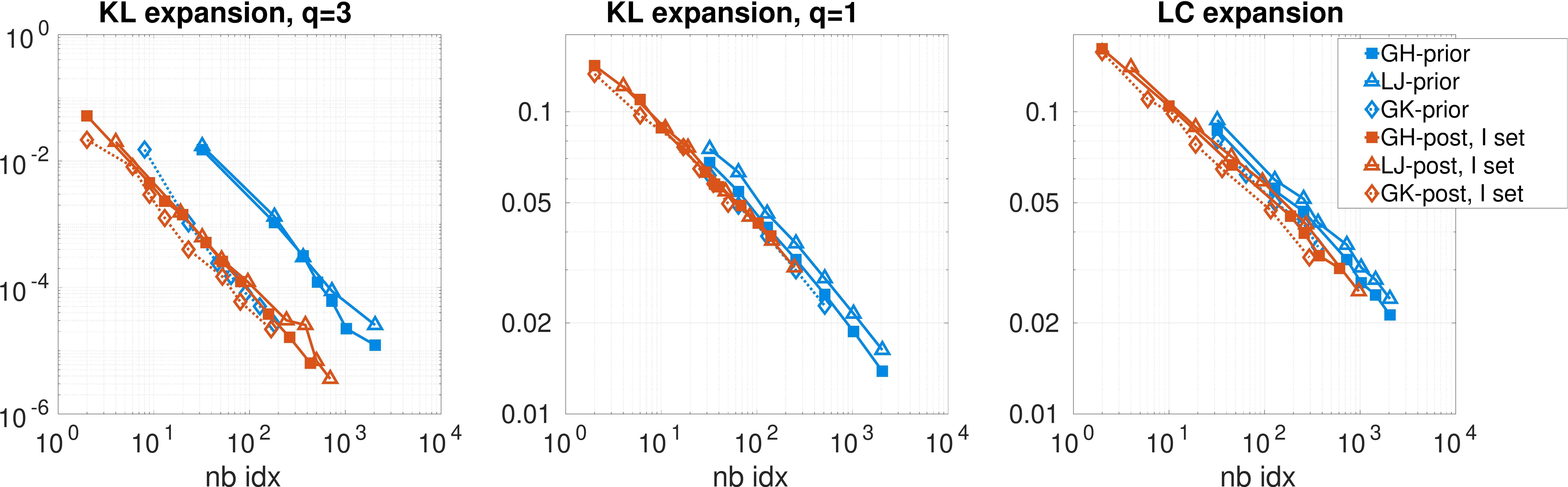}
  \caption{Top row: further analysis of influence of counting strategies in assessing the performance of Gaussian Leja, Genz--Keister, and Gauss--Hermite points.
    \rev{Middle and bottom rows:} plot of error versus number of indices in the sparse grid set for different test cases.
    \rev{The plots in these two rows are grouped in the same way as in Figure \ref{fig:GHvsLJ}.}}
  \label{fig:GHvsLJ_more}
\end{figure}

We begin the analysis \rev{with} the comparison of the performance of Gauss--Hermite, Gaussian Leja, and Genz--Keister points. 
To this end, we consider random fields of \rev{varying} smoothness, we choose an expansion (KL/LC) for each random field considered,
and we compute the sparse grid approximation of $u$ with 
the a-priori and a-posteriori dimension-adaptive sparse grid algorithm, with Gauss--Hermite, Gaussian Leja and Genz--Keister
points (i.e., 6 runs per choice of random field and associated expansion). 
Specifically, we consider three different random field \rev{expansions}, i.e., a KL expansion of the smoothed Brownian bridge with $q=3$, and a standard Brownian bridge ($q=1$) expanded with either KL or LC expansion, cf.\ again Examples \ref{exam:BB} and \ref{exam:smoothBB}. 
We compute the error in the full $L^2_\mu(\Gamma; H^1_0(D))$ norm again with a Monte Carlo sampling over 1000 samples of the random field, which has been verified to be sufficiently accurate for our purposes. 
These samples are generated considering a ``reference truncation level'' of the random field with 1000 random variables, which substantially exceeds the number of random variables active during the execution of the algorithms (which never involve more than a few hundred random variables).
In the first set of results, we report the convergence of the error with respect to the number of points in the grid. 
The manner of counting of the points is a subtle issue and can be done in various ways.
Here we consider the following different counting strategies:
\begin{description}
\item[\textbf{``incremental'':}] the number of points in the sparse grid $\Xi_\Lambda$ as defined in \eqref{equ:sparse_grid}, i.e., the points required to compute the application of $U_\Lambda$ as given in \eqref{equ:U_Lambda},
\item[\textbf{``combitec'':}] 
the number of points necessary for the combination technique representation of $U_\Lambda$ in \eqref{equ:combitec};
since $c(\vi;\Lambda)$ may be zero for some $\vi \in \Lambda$, we can omit the corresponding $U_\vi$ in \eqref{equ:combitec} and consider the possibly smaller combitec sparse grid $ \Xi^\text{ct}_\Lambda := \bigcup_{\vi \in \Lambda\colon c(\vi;\Lambda) \neq 0} \Xi^{(\vi)}$.
\end{description}

These strategies exhaust the counting strategies for the a-priori construction; note that these two counting schemes yield different values for non-nested points
(such as Gauss--Hermite), while they are identical for nested points (such as Gaussian Leja and Genz--Keister). 
For the a-posteriori construction, one should also further decide whether to apply these counting strategies including or excluding the indices in the margin of the current set (``I-set'' and ``G-set'' in the legend, respectively). \rev{Note that the ``I-set'' choice is more representative of the ``optimal index-set'' computed
by the algorithm, while the ``G-set'' is more representative of the actual computational cost incurred when running the algorithm.}

Results are reported in Figures \ref{fig:GHvsLJ} and \ref{fig:GHvsLJ_more}.
Throughout this section, we use the following abbreviations in the legend of \rev{the} convergence plots:
GH for Gauss--Hermite, LJ for Gaussian Leja, GK for Genz--Keister.
Figure \ref{fig:GHvsLJ} compares the performance of the three choices of points for the three choices of random
field \rev{expansions} and the \rev{two} sparse grid constructions mentioned earlier \rev{(a-posteriori/a-priori)},
in terms of $L^2_\mu$-error vs. number of collocation points.
Different colors identify different combination of grid constructions and counting
(\rev{light} blue for a-priori-incremental; red for a-posteriori-I-set-incremental; gray for a-posteriori-G-set-incremental).
\rev{The results for Gauss--Hermite points are indicated by \rev{solid} lines with \rev{square} filled markers,
those for Gaussian Leja points by \rev{solid} lines with empty \rev{triangle} markers,
and those for Genz--Keister by dashed lines with \rev{empty diamond markers}.}

The first and foremost observation to be made is that 
the Gaussian Leja performance is consistently better than Genz--Keister and Gauss--Hermite across algorithms (a-priori/a-posteriori) and test cases,
while Gauss--Hermite and Genz--Keister performance is essentially identical, in agreement with what reported e.g. in \cite{NobileEtAl2016,Chen2016}.
Only the Genz--Keister performance for the a-priori construction in the case $q=3$ is surprisingly good; we do not have an explanation for this,
and leave it to future research. Secondly, we observe that the a-priori algorithm performs
worse than the a-posteriori for $q=3$ (both considering the ``G-set'' and the ``I-set'' \rev{- left panel in the middle and bottom rows}),
while \rev{for the case $q=1$ it performs worse than the a-posteriori ``I-set'' but better than the a-posteriori
  ``G-set'' (regardless of type of expansion - mid and right panels in the central and bottom rows). This means that while there
  are better choices for the index set than a-priori one (e.g., the a-posteriori ``I-set''), these might be hard to derive,
  so that in practice it might be convenient to use the a-priori algorithm.}
This is in agreement with the findings reported in \cite{ErnstEtAl2018} and not surprising, given that in the case $q=1$ features a larger number
of random variables and therefore is harder to be handled by the a-posteriori algorithm. 

In Figure \ref{fig:GHvsLJ_more} we \rev{analyze in more detail} the relatively poor performance of Gauss--Hermite points.
In the top row we want to investigate whether the ``incremental''/``combitec'' counting (which we recall produces different
results only for Gauss--Hermite points) explains at least partially the gap between the Gauss--Hermite and the Gaussian Leja results in Figure \ref{fig:GHvsLJ}.
To this end, we focus on the a-posteriori ``I-set''.
For such grid and counting, we report the convergence curves from Figure \ref{fig:GHvsLJ} for both the Gauss--Hermite and the Gaussian Leja collocation points
and add in black with filled markers the ``combitec'' counting, which is more
favorable to Gauss--Hermite points. 
The plots show, however, that the counting method accounts for only a small fraction of the gap.

In the \rev{middle and} bottom \rev{rows} instead we investigate whether the set of multi-indices chosen by the algorithm also has an influence---in other words, could it be that because of the family of points, the algorithms are ``tricked'' into exploring less effective index sets?
To this end, we redo Figure \ref{fig:GHvsLJ} by showing the convergence with respect to the number of multi-indices in the set $\Lambda$,
instead of with respect to the number of points. The plots show that in this setting,
there is essentially no difference in performance between Gauss--Hermite, Gaussian Leja and Genz--Keister
points (again, excluding the case of Genz--Keister points for a-priori construction in the case $q=3)$,
which means that the sets obtained by the a-priori/a-posteriori algorithm, while different, are ``equally good''
in approximating the solution.\footnote{Incidentally, note that the a-priori algorithm
  doesn't take into account the kind of univariate nodes that will be used to build the sparse grids. Also note that
  of course the convergence of Gaussian Leja with respect to either number of points or number of multi-indices
  is identical, given that each multi-index adds one point.}
Thus, the consistent difference between Gaussian Leja, Genz--Keister and Gauss--Hermite nodes is really due to the nestedness \rev{of the  former two choices. 
Between the two choices of nested points, the Gaussian Leja points are more granular and easier to compute up to an arbitrary number: in conclusion, they appear to be a \rev{more suitable choice of collocation points} for the lognormal problem
\rev{in terms of accuracy versus number of points}.}


\subsection{KL vs. LC Expansion} \label{subs:KLEvsLCE}

The second set of tests aims at assessing whether expanding the random field over the wavelet basis (LC expansion)
brings any practical advantage in convergence of the sparse grid algorithm over using
the standard KL expansion.
Since from the previous discussion we know that Gaussian Leja nodes are more effective than Gauss--Hermite and Genz--Keister points,
we only consider Gaussian Leja points in this section.

\begin{figure}[t]
  \centering
  \includegraphics[width=\linewidth]{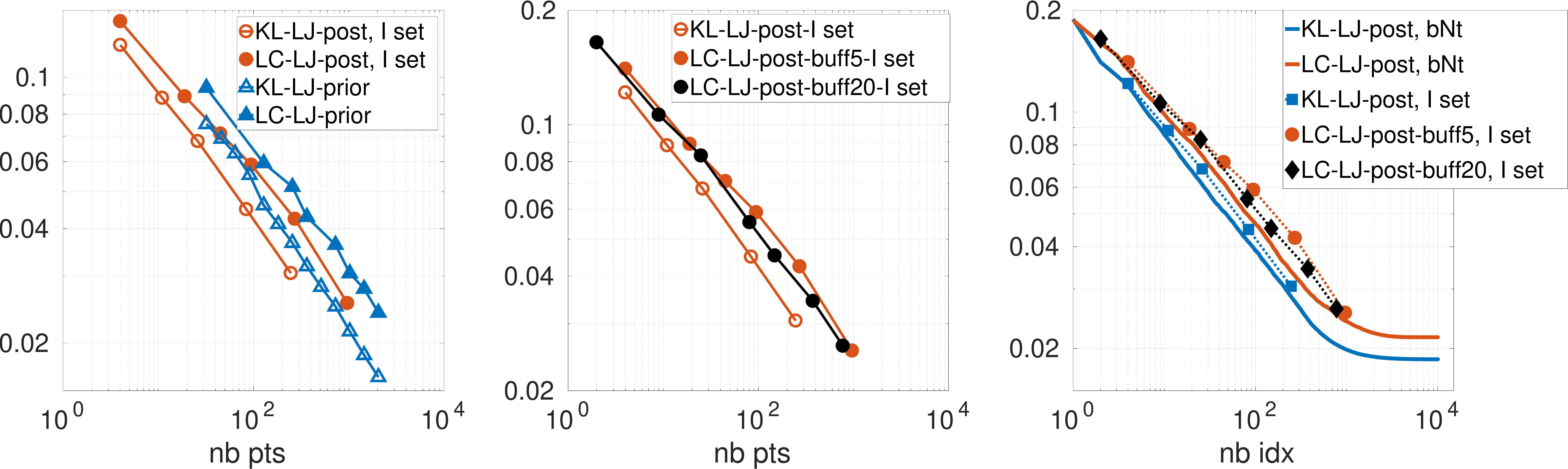}
  \caption{Comparison of performance for LC and KL expansions.}
  \label{fig:LCvsKL}
  \includegraphics[width=\linewidth]{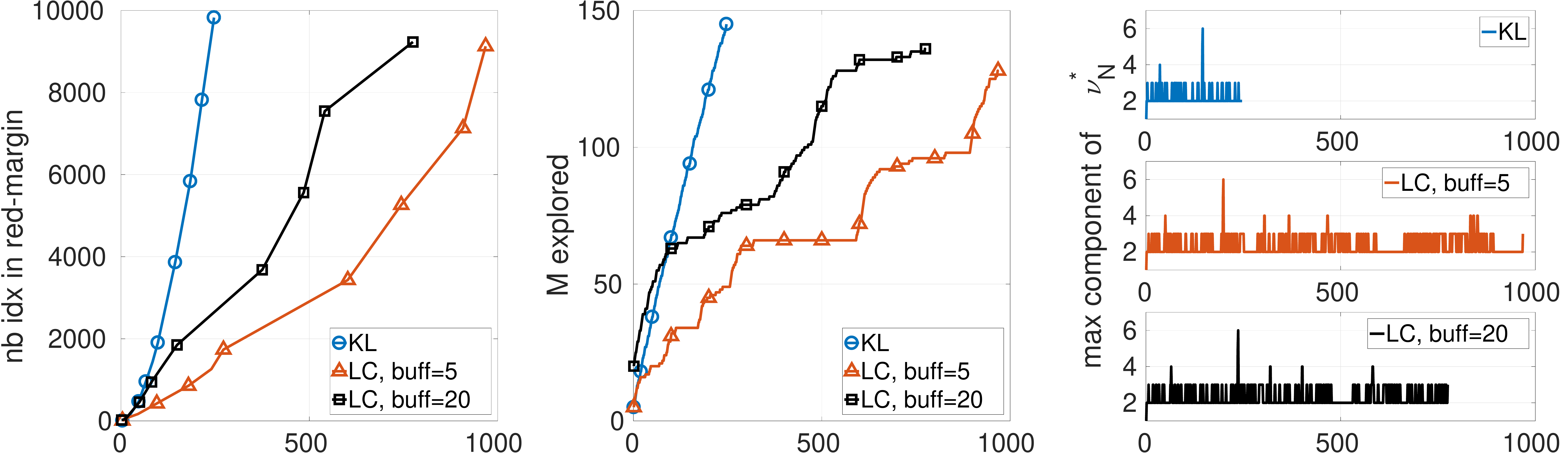}\\
  \caption{Evolution of the multi-index set $\Lambda$ for LC and KL expansions along
    iterations of the dimension-adaptive algorithm.}
  \label{fig:multiindx_set}
\end{figure}

Results are reported in Figure \ref{fig:LCvsKL}. In the left plot, we compare the convergence of the error versus number
of points for the a-priori and a-posteriori ``I-set'' for LC and KL expansion;
we employ the same color-coding as in Figure \ref{fig:GHvsLJ} (blue for prior construction, red for the ``I-set'' of
the a-posteriori construction), using filled markers for LC results and empty markers for \rev{KL} results.
The lines with \rev{filled} markers are always significantly \rev{above} the lines with empty markers, i.e.,
the convergence of the sparse grid adaptive algorithm is significantly faster for the
KL expansion than for the LC expansion.
This can easily be explained by the implicit ordering introduced by the
KL expansion in the importance of the random variables: because the modes of the KL are ordered in descending order
according to the percentage of variance of the random field they represent, they are already ordered in
a suitable way for the adaptive algorithm, which from the very start can explore informative directions of variance
(although the KL expansion is optimized for the representation of the input rather than for the output).
The LC expansion instead uses a-priori choices of the \rev{expansion basis functions} 
and in particular batches (of increasing cardinality) of \rev{those basis functions} 
are equally important (i.e., the wavelets at the same refinement level). On the other hand,
the adaptive algorithm explores random variables in the expansion order, which means that sometimes the algorithm
has to include ``\rev{unnecessary}'' modes of the LC expansion before finding those that really matter.

Of course, a careful implementation of the adaptive algorithm can, to a certain extent,  mitigate this issue.
In particular, increasing the size of the buffer of random variables (cf.\ the description at the beginning of Section~4) improves the performance of the adaptive algorithm. The default number of inactive random variables is 5---the convergence lines in the left plot are obtained in this way. 
In the middle plot we confirm that, as expected, increasing the buffer from 5 to 20 random variables improves the performance of the sparse grid approximation when applied to the LC case
(black line with filled markers instead of red line with filled markers).
Note, however, that a significant gap remains between the convergence of the sparse grid approximation for the LC expansion with a buffer of 20 random variables and the convergence of the sparse grid for the KL expansion. 
This means that not only does the buffer play a role, but the KL expansion is overall a more convenient basis to work with.

This aspect is further elaborated in the right plot of Figure~\ref{fig:LCvsKL}. 
Here we show the convergence of the sparse grid approximation for KL (5-variable buffer) and LC (either 5-variable or 20-variable buffer) against the number of indices in the sparse grids (dashed lines with markers), and compare this convergence against an estimate of the corresponding best-N-term (bNt) expansion of the solution in Hermite polynomials (full lines without markers); different colors identify different expansions. 
Of course, the convergence of the bNt expansion also depends on the LC/KL basis, therefore we show two bNt convergence curves.
The bNt was computed by converting the sparse grid into the equivalent Hermite expansion (see \cite{feal:compgeo,lever.eal:inversion} for details) and then rearranging the Hermite coefficients \rev{in order of decreasing magnitude}. 
The plot shows that the sparse grid approximation of the solution by KL expansion is quite close to the bNt convergence (blue lines), which means that there is not much room for ``compressibility'' in the sparse grid approximation. 
Conversely, the 5-variable-buffer sparse grid approximation of the problem with LC expansion
is somehow far from the bNt (red lines) and only the 20-variable-buffer (black dashed line) gets reasonably close: this means that the 5-variable-buffer is ``forced'' to add to the approximation ``useless'' indices merely because the ordering of the variables in the LC expansion is not optimal and the buffer is not large enough.

Finally, we report in Figure \ref{fig:multiindx_set} some performance indicators for the construction of the index set for the KL and LC cases, which offer further insight towards explaining the superior KL performance.
The figure on the left shows the growth of the size of the outer margin of the dimension-adaptive algorithm at each iteration, where we recall that one iteration is defined as the process of selecting one index from the outer margin and evaluating the error indicator for all its forward neighbors; this in particular means that the number of PDE solves per iteration is not fixed. 
All three algorithms (KL, 5-variable-buffer LC and 20-variable-buffer LC) stop after 10,000 PDE solves.
KL displays the fastest growth in the outer margin size, followed by LC20 and then LC5, which is perhaps counter-intuitive; on the other hand, the more indices are considered, the more likely it is to find ones ``effective'' in reducing the approximation error. 
The figure in the center shows the growth in the number of explored dimensions: again, KL has the quickest and steadiest growth, which means that the algorithm favors
adding new variables over exploring those already active. 
This might be again counter-intuitive, but there is no contradiction between this \rev{observation} and the superior performance of KL: the point here is actually
precisely the fact that the LC random variables are not \rev{conveniently} sorted, so the algorithm is \rev{obliged} to explore those already available rather than adding new ones; this is especially visible
for the LC5 case, which displays a significant plateau in the growth in the number of variables in the middle of the algorithm execution. 
The three plots on the right finally show the largest component of multi-index $\mathbf{\nu}_N^*$
  that has been selected from the reduced margin at iteration $N$ for the three algorithms (from the top: KL, LC5, LC20):
  a large maximum component means that the algorithm has favored exploring variables already activated,
  while if the maximum component is equal to 2 the algorithm has activated a new random variable
  (\rev{indices} start from 1 in the Sparse Grids Matlab Kit). 
Most of the values in these plots are between 2 and 3,
  which again shows that the algorithms favor adding new variables rather than exploring those already available.
Finally, we mention (plot omitted for brevity) that despite the relatively large number of random variables
  activated, each tensor grid in the sparse grid construction is at most 4-dimensional\footnote{\rev{In other words, out of the $M$ random variables considered, only four are simultaneously activated to build the tensor grids---which four of course depends on each tensor grid.}}, which means that \rev{interactions between five or more of the random variables appearing in the KL or LC expansion, respectively, are considered negligible by the algorithm. }
%

\section{Conclusions}\label{sec:concl}
In this contribution we have investigated some practical choices related to the numerical approximation of random elliptic PDEs with lognormal diffusion coefficients by sparse grid collocation methods.
More specifically, we discussed two issues, namely 
a) whether it pays off from a computational point of view to replace the classical
Karhunen–Loève expansion of the log-diffusion field with the L\'evy–Ciesielski expansion, as  advocated in [2] for theoretical purposes and 
b) what type of univariate interpolation node sequence should be used in the sparse grid construction, choosing among Gauss–-Hermite, Gaussian Leja and Genz–-Keister points.  
\rev{Following a brief digression into the issue of } convergence of
interpolation and quadrature of univariate and multivariate functions based on these three classes of nodes, we compared the performance of sparse grid collocation for the approximate solution of the lognormal random PDEs in a number of different cases.  
The \rev{computational experiments} suggest that Gaussian Leja collocation points, \rev{due to their approximation properties, granularity and nestedness, are the superior choice} for the sparse grid approximation of the random PDE under consideration, and that the Karhunen–Lo\`eve expansion \rev{offers a computationally  more effective parametrization of the input random field}
than the L\'evy–Ciesielski expansion.

\begin{acknowledgement}
The authors would like to thank Markus Bachmayr and Giovanni Migliorati for helpful discussions and Christian J\"ah for Proposition \ref{propo:jaeh}.
Bj\"orn Sprungk is supported by the DFG research project 389483880.
Lorenzo Tamellini has been supported by the GNCS 2019 project
  ``Metodi numerici non-standard per PDEs: efficienza, robustezza e affidabilit\`a''
  and by the PRIN 2017 project ``Numerical Analysis for Full and Reduced Order Methods for the efficient
  and accurate solution of complex systems governed by Partial Differential Equations''.
\end{acknowledgement}

\section*{Appendix} \label{ErnstEtAl:sec:app}
\addcontentsline{toc}{section}{Appendix}
We show that the Karhunen--Lo\`eve expansion of the Brownian bridge discussed in Example \ref{exam:BB} does not satisfy the conditions of Theorem \ref{theo:Bachmayr_reg} for $p>0$. 
To this end, we first state

\begin{proposition}\label{propo:jaeh}
Let $(b_m)_{m\in\bbN}$ be a monotonely decreasing sequence of real numbers with $\lim_{m\to\infty} b_m = 0$.
Then for any $\theta \in [0,2\pi]$ we have
\[
	\sum_{m\geq1} b_m \sin (m \theta) < \infty.
\]
\end{proposition}
\begin{proof}
Dirichlet's test for the convergence of series implies the statement if there exists a constant $K<\infty$ such that
\[
	\left| \sum_{m=1}^M \sin (m \theta) \right| \leq K \qquad \forall M\in\bbN.
\]
Now, Lagrange's trigonometric identity tells us that
\[
	\sum_{m=1}^M \sin (m \theta)
	=
	\frac 12 \cot(0.5\theta) - \frac{\cos\left((M+0.5) \theta \right)}{2\sin(0.5\theta)},
	\qquad
	\theta \in (0,2\pi).
\]
Hence, since $\sin (m 0) =  \sin (m 2\pi) = 0$ the statement follows easily.
\end{proof}

\begin{proposition}\label{propo:BB_KLE}
Given the Karhunen--Lo\`eve expansion of the Brownian bridge as in \eqref{equ:BB_KLE}, the function
\[
	k_{\vtau}(x) := \sum_{m=1}^\infty \tau_m \frac{\sqrt2}{\pi m} \sin(m\pi x),
	\qquad
	x\in D=[0,1],
\]
is pointwise well-defined for $\tau_m = m^{1/q}$ with $q>1$ in which case $(\tau_m^{-1})_{m\in\bbN} \in \ell^p(\bbN)$ for any $p>q>1$.
However, assuming that $k_{\vtau}\colon [0,1]\to\bbR$ is well-defined for a sequence $\vtau = (\tau_m)_{m\in\bbN}$ with $(\tau^{-1}_m)_{m\in\bbN} \in \ell^p(\bbN)$ for a $p \leq 2$, then $k_{\vtau} \notin L^\infty(D)$.
\end{proposition}
\begin{proof}
The first statement follows by Proposition \ref{propo:jaeh} and $\frac{\sqrt2}{\pi m} \tau_m = C m^{1/q-1} \to 0$ as $m\to \infty$.
The second statement follows by contracdiction. Assume that $k_{\vtau} \in L^\infty(D)$, then also $k_{\vtau} \in L^2(D)$ and via $\|k_{\vtau}\|_{L^2(D)} = \frac 1{\pi^2} \sum_{m=1}^\infty \frac{\tau^2_m}{m^2}$ we have that $\tau^2_m \leq cm$ for a $c \geq 0$---otherwise $\|k_{\vtau}\|_{L^2(D)} = +\infty$. 
Thus, $\tau_m^{-p} \geq c^{-p/2} m^{-p/2}$ and since $\sum_{m\geq 1}m^{-p/2} < +\infty$ if and only if $p > 2$, we end up with $(\tau_m^{-1})_{m\in\bbN} \notin \ell^2(\bbN)$.
\end{proof}

For values $p>2$ we provide the following numerical evidence: we choose 
$\tau_m = m^{1/p}$, i.e., $(\tau_m^{-1})_{m\in\bbN}\in \ell^{p+\epsilon}(\bbN)$, $\epsilon>0$, and compute the values of the function 
$\kappa_{\vtau}(x)$ as given in Proposition \ref{propo:BB_KLE} in a neighborhood of $x=0$ numerically. 
The reason we are interested in small values of $x$ is the fact that $\kappa_{\vtau}(x)$, $x\neq 0$, can be bounded by $\frac 12 \cot(0.5\pi x) + \frac{1}{2\sin(0.5\pi x)}$ by means of Proposition \ref{propo:jaeh}.
Thus, we expect a blow-up for small values of $x$.
Indeed, we observe numerically that $\kappa_{\vtau}(x)$ for $\tau_m = m^{1/p}$ behaves like $x^{-1/p}$ for small values of $x>0$, see Figure \ref{fig:Kappa_tau}. 
This implies that $\kappa_{\vtau}$ is unbounded in a neighborhood of $x=0$ for any of the above choices of $\tau_m$ and, therefore, does not satisfy the conditions of Theorem \ref{theo:Bachmayr_reg}.

\begin{figure}
\begin{minipage}[c]{0.5\textwidth}
\centering
\includegraphics[width=\textwidth]{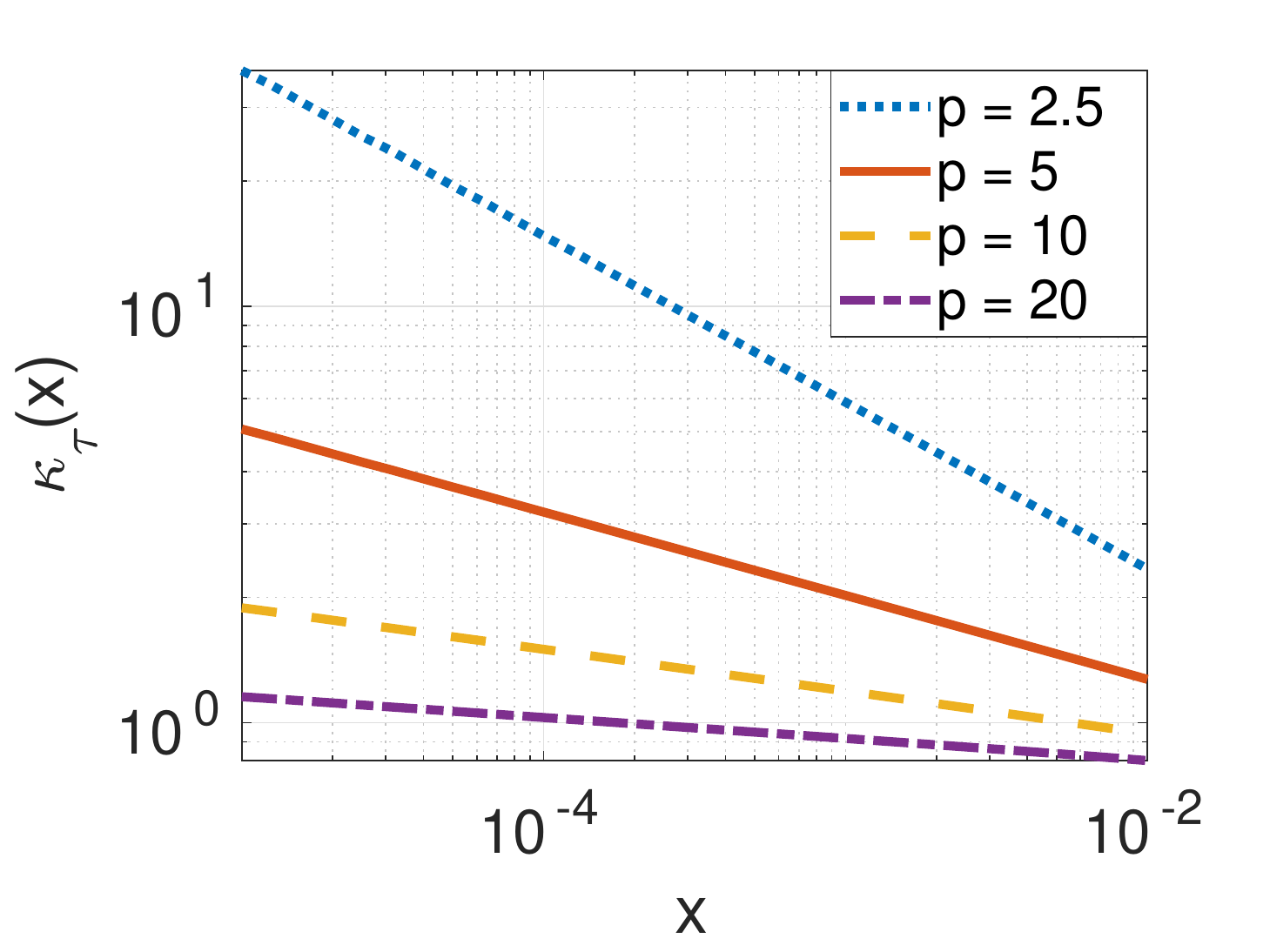}
\end{minipage}
\begin{minipage}[b]{0.49\textwidth}
\phantom{a}
\caption{Growth of $\kappa_{\vtau}(x)$ as given in Proposition \ref{propo:BB_KLE} for decaying $x\to 0+$ and choices $\tau_m = m^{1/p}$ with various values of $p$---the observed growth matches $x^{-1/p}$.}
\label{fig:Kappa_tau}
\end{minipage}
\end{figure}

\bibliography{literature}
\bibliographystyle{abbrv}
\end{document}